\documentclass{article}
\usepackage[dvipsnames]{xcolor}

\usepackage[colorlinks=true,
            linktocpage=true,
            citecolor=RubineRed,
            linkcolor=RoyalBlue]{hyperref}

\usepackage{multicol}
\usepackage{multirow}
\usepackage{amsmath}
\usepackage{amssymb}
\usepackage{amsthm}

  \theoremstyle{definition}
    \newtheorem{theorem}{Theorem}[section]

    \newtheorem{lemma}[theorem]{Lemma}
    \newtheorem{proposition}[theorem]{Proposition}
    \newtheorem{definition}[theorem]{Definition}
    \newtheorem{corollary}[theorem]{Corollary}
    \newtheorem{example}[theorem]{Example}
  \theoremstyle{plain}

\usepackage[]{graphicx}

\usepackage{doi}
\usepackage[a4paper,margin=2.5cm]{geometry}
\usepackage{stmaryrd}
\usepackage{tikz}
\usepackage{tikz-cd}
  \usetikzlibrary{positioning,calc,arrows,arrows.meta,decorations.pathmorphing}
  \tikzcdset{arrow style=tikz, diagrams={>=stealth'}}
\tikzset{modal/.style={>=stealth?,shorten >=1pt,shorten <=1pt,auto,node distance=1.5cm, semithick}, world/.style={circle,draw,minimum size=0.5cm,fill=gray!15}, point/.style={circle,draw,inner sep=0.5mm,fill=black}, reflexive above/.style={->,loop,looseness=7,in=120,out=60}, reflexive below/.style={->,loop,looseness=7,in=240,out=300}, reflexive left/.style={->,loop,looseness=7,in=150,out=210}, reflexive right/.style={->,loop,looseness=7,in=30,out=330}}
\usepackage{booktabs}
\usepackage{enumerate}
\usepackage{proof}
\usepackage{bussproofs}

\usepackage{mathrsfs}
\usepackage[cal=txupr,scr=boondoxupr,scrscaled=1.050]{mathalfa}

\newcommand{\AAL}{AAL}

\newcommand{\mo}[1]{\mathfrak{#1}}  
\newcommand{\lan}[1]{\mathbf{#1}}   
\renewcommand{\log}[1]{\mathsf{#1}}   
\newcommand{\alg}[1]{#1}   
\newcommand{\clalg}[1]{\mathcal{#1}}   

\newcommand{\excl}{\mathrel{%
  \hspace{.1ex}
  \begin{tikzpicture}[baseline=-.57ex, line width=.130ex]
    \draw[-] (-0.1ex,0) --(1.5ex,0);
    \draw[-, line width=.01ex, fill=black]
             (1.35ex,0) -- (1.84ex, .48ex)
                       -- (1.91ex ,.418ex)
                       -- (1.55ex,   0ex)
                       -- (1.91ex ,-.418ex)
                       -- (1.84ex,-.48ex)
                       -- (1.35ex,0ex);
  \end{tikzpicture}
\hspace{.1ex}}}
\newcommand{\wneg}[0]{{\sim}}
\newcommand{\DN}[0]{\neg\wneg}

\newcommand{\lanbi}{\lan{L}_{BI}}

\newcommand{\GHrule}[1]{(\text{#1})}
\newcommand{\wderiv}[2]{{#1}\vdash_{\log{w}}{#2}}
\newcommand{\nowderiv}[2]{{#1}\not\vdash_{\log{w}}{#2}}
\newcommand{\sderiv}[2]{{#1}\vdash_{\log{s}}{#2}}
\newcommand{\nosderiv}[2]{{#1}\not\vdash_{\log{s}}{#2}}
\newcommand{\ideriv}[2]{{#1}\vdash_{\log{i}}{#2}}
\newcommand{\noideriv}[2]{{#1}\not\vdash_{\log{i}}{#2}}
\newcommand{\wbil}[0]{\log{wBIL}}
\newcommand{\sbil}[0]{\log{sBIL}}

\newcommand{\biha}{\mathcal{BHA}}
\newcommand{\alambda}{\dot\lambda}
\renewcommand{\atop}{\dot\top}
\DeclareMathOperator{\abot}{\dot\bot}
\DeclareMathOperator{\aland}{\dot\land}
\DeclareMathOperator{\alor}{\dot\lor}
\DeclareMathOperator{\ato}{\dot\to}
\DeclareMathOperator{\aneg}{\dot\neg}
\DeclareMathOperator{\aexcl}{\dot\excl}
\DeclareMathOperator{\awneg}{\dot\wneg}
\DeclareMathOperator{\aDN}{\aneg\awneg}
\DeclareMathOperator{\aleftrightarrow}{\dot\leftrightarrow}

\newcommand{\eqprv}[1]{\llbracket #1 \rrbracket}
\newcommand{\eqcla}[2]{\eqprv{Form}^{\log{#1}}_{#2}}
\newcommand{\LT}[2]{\mathbb{LT}^{\log{#1}}_{#2}}

\newcommand{\pt}{\vdash_{pt}}
\newcommand{\pdt}{\vdash_{pdt}}

\newcommand{\form}{Form}

\newcommand{\val}{v}
\renewcommand{\int}[1]{\hat{#1}}

\newcommand{\cdash}[1][]{\mathrel{\kern.1ex\text{%
  \tikz[baseline=-.81ex, line width=.1ex, line cap=round, scale=1.1]
    {\draw (0ex,-.75ex) -- (0ex,.75ex);
     \draw (0ex,0ex) -- (1ex,0ex);
     \draw (-.4ex,0ex) arc(180:235:.91ex);
     \draw (-.4ex,0ex) arc(180:125:.91ex);}$_{#1}$}}\kern.1ex}
\newcommand{\locmodels}{\cdash[l]}
\newcommand{\globmodels}{\cdash[g]}
\newcommand{\nolocmodels}{\not\cdash[l]}
\newcommand{\noglobmodels}{\not\cdash[g]}
\makeatletter
\providecommand*{\Dashv}{%
  \mathrel{%
    \mathpalette\@Dashv\vDash
  }%
}
\newcommand*{\@Dashv}[2]{%
  \reflectbox{$\m@th#1#2$}%
}
\makeatother
\newcommand{\zz}{\lessgtr}

\newcommand{\congr}{\approx}
\newcommand{\bisim}{\rightleftharpoons}

\renewcommand{\iff}{\quad\text{iff}\quad}
\renewcommand{\phi}{\varphi}
\DeclareMathOperator{\Prop}{Prop}

\newcommand{\N}{\mathbb{N}}

\DeclareMathOperator{\dto}{\mathrel{\kern1pt\dot\to\kern1pt}}
\DeclareMathOperator{\dlor}{\mathrel{\kern.5pt\dot\lor\kern.5pt}}
\DeclareMathOperator{\dland}{\mathrel{\kern.5pt\dot\land\kern.5pt}}

\DeclareMathOperator{\dleftrightarrow}{\mathrel{\kern1pt\dot\leftrightarrow\kern1pt}}

\newcommand\Item[1][]{%
  \ifx\relax#1\relax  \item \else \item[#1] \fi
  \abovedisplayskip=0pt\abovedisplayshortskip=0pt~\vspace*{-\baselineskip}}

\def\BaseUrl{https://ianshil.github.io/algbiint}
\newcommand{\coqdoc}[1]{\href{\BaseUrl/#1}{\raisebox{-0.6mm}{\includegraphics[height=0.8em]{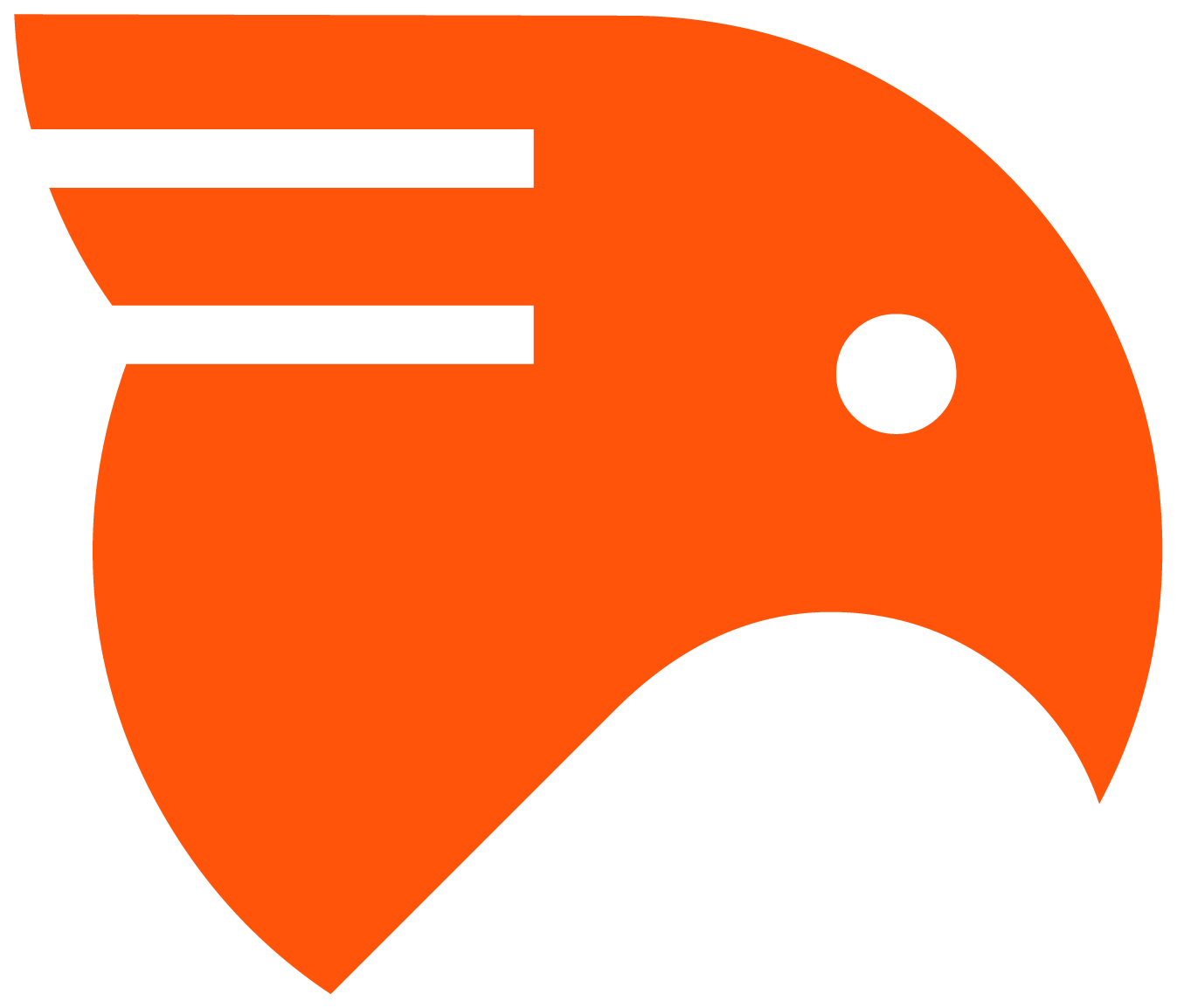}}}}

\title{Bi-intuitionistic logics through the abstract algebraic logic lens}

\author{Jonte Deakin \\
    \footnotesize{\texttt{u7480977@anu.edu.au}} \\
    \footnotesize{The Australian National University, Canberra, Ngunnawal \& Ngambri Country, Australia} \\
    \\
    Ian Shillito \\
    \footnotesize{\texttt{i.b.p.shillito@bham.ac.uk}} \\
    \footnotesize{University of Birmingham, Birmingham, United Kingdom}}
    \date{}

\begin{document}

\maketitle        

\begin{abstract}
Since the discovery of critical mistakes in Rauszer's work on bi-intuitionistic logics,
solid foundations for these have progressively been rebuilt.
However, the algebraic treatment of these logics has not yet been
tended to.
We fill this gap by algebraically analysing the bi-intuitionistic logics 
$\wbil$ and $\sbil$.
Given that these logics are only distinguished as consequence relations,
and not as sets of theorems (hence the conflation in Rauszer's work),
the algebraic tools we use are tailored to the treatment of such relations.
We mainly inspect these logics through the lens of abstract algebraic logic,
but we also provide an alternative algebraic analysis of $\wbil$ and $\sbil$ as
logic preserving degrees of truth and truth, respectively.
Our results pertaining to $\wbil$ and $\sbil$ are formalised in the interactive theorem
prover Rocq.

\vspace{0.2cm}
\noindent \textbf{Keywords:} abstract algebraic logic, bi-intuitionistic logic, interactive theorem prover Rocq, algebraizability
\end{abstract}


\section{Introduction}

Bi-intuitionistic logic extends intuitionistic logic with the \emph{exclusion}
operator $\excl$, dual to intuitionistic implication $\to$.
This extension is mathematically natural, as with $\excl$ a symmetry is regained in the language
given that each operator now has a dual: 
$\top$ has $\bot$, $\land$ has $\lor$, $\to$ has $\excl$ and
even $\neg\phi:=\phi\to\bot$ has $\wneg\phi:=\top\excl\phi$.
Despite being simple, this addition is surprisingly impactful as key constructive properties of intuitionistic
logic, like the disjunction property, are not retained in bi-intuitionistic logic.
This failure is notably due to the holding of a bi-intuitionistic version of the 
\emph{law of excluded middle}: $\phi\lor\wneg\phi$.
Still, the logic obtained \emph{conservatively} extends intuitionistic logic, which implies
that the addition of the exclusion operator does not yet cause a collapse to classical logic.

Historically, bi-intuitionistic logic was developed by Cecylia Rauszer in a series of articles
in the 1970s~\cite{Rau74alg,Rau74calc,Rau76,Rau77craig} leading to her Ph.D.~thesis~\cite{Rau80}.
While appearances of bi-intuitionistic can be detected prior to these articles,
notably in 1942 in Moisil's work~\cite{Moi42} and in 1971 in Klemke's work~\cite{Kle71}, 
the breadth and foundational nature of Rauszer's work advocate for her place
as founder of this logic.
In particular, she studied this logic under a wide variety of aspects:
algebraic semantics, axiomatic calculus, sequent calculus, and Kripke semantics. 
Unfortunately, through time a variety of mistakes have been detected in her work.
First, Crolard proved in 2001 that bi-intuitionistic logic is not complete for
the class of rooted frames~\cite[Corollary 2.18]{Cro01}.
However, the canonical model constructions of Rauszer are rooted, hence a contradiction.
Second, Pinto and Uustalu found in 2009~\cite{PinUus09} a counterexample to 
Rauszer's claim of admissibility of cut for the sequent calculus she designed~\cite[Theorem 2.4]{Rau74calc}.
Finally, Gor\'{e} and Shillito in 2020~\cite{GorShi20} detected confusions in Rauszer's work about the holding of the deduction theorem in bi-intuitionistic logic, as well as issues in her completeness proofs.

At the root of the issues in Rauszer's work is a conflation: two \emph{distinct} logics, 
understood as consequence relations, were identified in her treatment of bi-intuitionistic logic~\cite{GorShi20}.
More precisely, Gor\'{e} and Shillito exhibited a natural split between a \emph{weak} and
\emph{strong} bi-intuitionistic logic,
repesctively called $\wbil$ and $\sbil$.
They showed that 
(1) $\sbil$ is an extension of $\wbil$, 
(2) $\wbil$ satisfies the traditional deduction theorem while $\sbil$ only satisfies a modified version of it,
and that (3) $\wbil$ and $\sbil$ respectively captured the \emph{local} and \emph{global} consequence
relations of the Kripke semantics for the bi-intuitionistic language.
This situation perfectly mirrors a known splitting phenomenon in modal logic,
where the \emph{local} and the \emph{global} modal logics satisfy (1), (2) and (3)
in the same way $\wbil$ and $\sbil$ do.

Once this conflation is noted, the foundations laid by Rauszer can be disentangled and rectified. 
This is effectively what was performed by Gor\'{e} and Shillito~\cite{GorShi20} for the axiomatic calculi and Kripke semantics of $\wbil$ and $\sbil$.
This work was fully formalised in the interactive theorem prover Rocq (formerly Coq~\cite{Coq820})
in the second author's PhD thesis~\cite{Shi23}.
Similar results connecting the local \emph{first-order} bi-intuitionistic logic and a Kripke semantics
with constant domains were formalised by Kirst and Shillito~\cite{KirShi25}.
Further, a constructive reverse analysis of the results of soundness and completeness for $\wbil$
with respect to the Kripke semantics was led in Rocq by Shillito and Kirst~\cite{ShiKir24}.
Additionally, prior works on proof systems for bi-intuitionistic logic~\cite{Gor00,Cro01,Wan08,PinUus09,Tra17,KowOno17,GorPosTiu08,BuiGor07}
can be assessed through the distinction between the two logics as pertaining to the local logic $\wbil$.
In the first-order case, a recent work studied the local bi-intuitionistic logic
via polytree labelled sequent calculi~\cite{LyoShiTiu25}.
To date, a major pillar has been omitted in this revision process: the study of logic via algebras.
In this paper, we fill this gap by separating and studying
$\wbil$ and $\sbil$ through algebraic semantics.

Given that $\wbil$ and $\sbil$ can only be distinguished as consequence relations,
and not sets of theorems, we need algebraic tools tailored to such relations.
Note that traditional tools of algebraic semantics may not be helpful,
as they tend to define logics as sets of theorems and most often as the set of
formulas always interpreted as the top element of the algebras of a given class.

\emph{Abstract algebraic logic}, hereafter called \AAL, is an approach to algebraic semantics
founded on the treatment of logics as consequence relations.
This field has developed over the last half century, with early
works by Rasiowa~\cite{Ras87} and Czelakowski~\cite{Cze81}.
Arguably, the start of what we know today as \AAL~begins with a 1986 work by Blok and Pigozzi~\cite{BloPig86},
which they further developed in following years~\cite{BloPig89,BloPig92}.
From there on, this field was expanded by many and flourished into
modern abstract algebraic
logic~\cite{Mor24,CinNog13,CinNog21,FonJanPig03,LavPre21,Raf13}.%
\footnote{A place where one can find a much more detailed
history of abstract algebraic logic is the introduction to Font's book on the
topic~\cite{Fon16}.}
With the tools of abstract algebraic logic, ties between logics and algebras
can be categorised according to their strength, giving rise to the so-called \emph{Leibniz hierarchy}.
High enough in the hierarchy, some ties are so strong that they form \emph{bridges}
between a logic and a class of algebras, as they allow one to transfer properties of
the class of algebras into properties of the logic and vice-versa~\cite{BloPig89,CzePig99,KowMet19}.

In what follows we situate the logics $\wbil$ and $\sbil$ in the Leibniz hierarchy.
We show that $\sbil$ stands high in this hierarchy, as it is 
implicative and finitely algebraizable with respect to the class of bi-Heyting algebras.
The algebraizability of $\sbil$ allows us to use bridge theorems providing insights on
properties of $\sbil$ and of bi-Heyting algebras.
In comparison, $\wbil$ is located much lower in the hierarchy as it is only equivalential,
and not even finitely so.
The contrast between $\sbil$ and $\wbil$ is stark, as the former is tightly connected
to bi-Heyting algebras while we prove the latter to have no algebraic semantics (in the sense of \AAL)
over these algebras.

While striking, the distinction between the two logics is clearly visible in the algebraic setting
only thanks to the general framework of \AAL,
which crucially treats logics as consequence relations. 
Such tools, developed in the late 1980s, were not available to Rauszer at the time
of her work on bi-intuitionistic logic in the 1970s.
So, as a by-product of our results, we would like to offer the following historical analysis
of the confusions in Rauszer's works:
they are due to the absence, at the time, of a well-developed adequate framework to
study bi-intuitionistic logics as consequence relations.

To stay clear from the turmoils of the past,
we continue the line of work of the second author and formalise our results in Rocq,
using the full strength of classical logic in the metalogic.
Unfortunately, many results we rely on come from \AAL, which is
an extremely abstract and general field.
Formalising this field would take years of community effort,
a task we have not tried to achieve in this paper.
Instead, we assume general \AAL~results to hold and formalise
results which pertain in particular to the logics $\wbil$ and $\sbil$
and to the class of bi-Heyting algebras.
For clarity, each formalised results is flagged by a clickable Rocq logo~\coqdoc{}
leading to its formalisation in Rocq.
Our entire formalisation can be conveniently explored on a web browser
(\url{https://ianshil.github.io/algbiint/toc.html})
or downloaded
(\url{https://github.com/ianshil/algbiint}).

We start in Section~\ref{sec:aal} by introducing the elements of \AAL~necessary
to our work.
More precisely, we provide notions of universal algebra and discuss the Leibniz hierarchy.  
Then, we lay in Section~\ref{sec:biint} the syntactic, axiomatic and Kripke semantic
foundations of bi-intuitionistic logics.
Section~\ref{sec:leibhier} is the core of our paper and is devoted to the situation of
the logics $\sbil$ and $\wbil$, with the former treated first and the latter second.
We leverage this analysis in the Leibniz hierarchy in Section~\ref{sec:bridge}
by using bridge theorems to harvest results about $\sbil$.
Before concluding in Section~\ref{sec:concl}, we make a step aside \AAL~in
Section~\ref{sec:preserv} and
provide an alternative algebraic analysis of both logics through the notions
of logic \emph{preserving truth} and \emph{preserving degrees of truth}.


\section{Abstract algebraic logic}\label{sec:aal}

In this section we recall the basic definitions and theorems from \AAL~which we
use in the sequel.
We do not aim to provide a complete introduction or explanation of \AAL:
for this purpose, we refer to the textbook by Font~\cite{Fon16} and Moraschini's
monograph~\cite{Mor24}.

\subsection{Universal Algebra}

\AAL~uses universal algebra for the algebraic part of algebraic logic.
So, we introduce the elements of universal algebra needed for our study.%
\footnote{The textbook by Burris and Sankappanavar~\cite{BurSan81} is a 
good reference for an introduction to universal algebra.}

\begin{definition}
  \label{def:language}
  A \emph{language} is a pair $\lan{L} = (Op,ar)$ where
  $Op$ is a set of operator symbols, and
  $ar : Op \to\N$ is a function assigning an arity to each operator symbol.
\end{definition}

\begin{definition}
  \label{def:univ-algebra}
  An \emph{algebra} $\alg{A}$ of type $\lan{L}$ is a pair $(X,\{\lambda^\alg{A} \mid \lambda \in Op\})$ such that 
  $X$ is a set, and for all $\lambda$, $\lambda^\alg{A}$ is a function $\lambda^\alg{A} : \alg{A}^n \to \alg{A}$ 
  for $n = ar(\lambda)$. 
\end{definition}

We often abuse notation and write $a\in\alg{A}$, and $B\subseteq\alg{A}$ and $f:Y\to\alg{A}$ to mean
$a\in X$, and $B\subseteq X$ and $f:Y\to\alg{A}$ for an algebra $\alg{A}=(X,\{\lambda^\alg{A} \mid \lambda \in Op\})$.

Next, we define congruences, which are well behaved equivalence
relations used to define quotient algebras.

\begin{definition}
  \label{def:congruences}
  A \emph{congruence relation} on an algebra $\alg{A}$ of type $\lan{L}$
  is a relation $\congr$ satisfying the following.
  \begin{center}
  \begin{tabular}{r l}
    (Refl) & $a \congr a$ \\
    (Symm) & If $a \congr b$ then $b \congr a$ \\
    (Trans) & If $a \congr b$ and $b \congr c$ then $a \congr c$ \\
    (Cong) & If $a_1 \congr b_1$ and ... and $a_n \congr b_n$, then for any $\lambda$ with arity $n$,
      $\lambda^\alg{A} a_1 ... a_n \congr \lambda^\alg{A} b_1 ... b_n$
  \end{tabular}
  \end{center}
  The set of congruences on an algebra $\alg{A}$ will be written $Con(\alg{A})$.
\end{definition}

The following definitions pertaining to congruences will help us situate the logic $\wbil$ in the Leibniz hierarchy.

\begin{definition}
 Given a congruence relation $\congr$ and an algebra $\alg{A}$, we define the \emph{quotient algebra} $\alg{A}/\!\!\congr$,
  as having the underlying set as the quotient of the set by $\congr$ and the operators 
  $\lambda^{\alg{A}/\congr}([a_1],...,[a_n]) = [\lambda^\alg{A}
  a_1...a_n]$\footnote{This is a well defined algebra by the (Cong) property of
    the congruence relation.}

  Given a class of algebras $\clalg{K}$ and an arbitrary algebra $\alg{A}$, let 
  $Con_\clalg{K}(\alg{A}) = \{\congr \ \in Con(\alg{A}) \mid \alg{A}/\!\!\congr\,\in \clalg{K}\}$
  be the \emph{relative congruences} of $\alg{A}$ with respect to $\clalg{K}$.
\end{definition}

We now draw our attention to formulas generated from a language.

\begin{definition}
  Let $\Prop$ be a unique countably infinite set of propositional variables,
  and $\lan{L}=(Op,ar)$ a language.
  The set $\form_\lan{L}$ of formulas of $\lan{L}$ is inductively defined through the following
  clauses:
  \begin{enumerate}
    \item if $p\in\Prop$ then $p\in\form_\lan{L}$;
    \item for all $\lambda \in Op$ with $ar(\lambda) = n$,
    	if $a_i \in \form_\lan{L}$ for all $1\leq i\leq n$, then $\lambda a_1...a_n \in \form_\lan{L}$.
  \end{enumerate}
\end{definition}

  When $\lambda$ is a binary connective we usually use infix notation,
  i.e.~we write $a_1 \lambda a_2$ for $\lambda a_1a_2$.
  We omit the subscript of $\form_\lan{L}$ when $\lan{L}$ is clear from context.
  We use the letters $p,q,r$ to denote elements of $\Prop$, and $\phi,\psi,\chi,\theta,\delta,\epsilon,\gamma,\eta,\zeta$ to
  denote elements of $\form_\lan{L}$, and $\Gamma,\Delta$ for sets of formulas.
  We also write $\overline{p}$ for a sequence of proposition letters.
  Often, a formula $\phi$ is written as $\phi(\overline{p})$ to indicate that
  the proposition letters occurring in it are among those in the sequence
  $\overline{p}$. 
  We also use this convention for sets of formulas and write $\Gamma(\overline{p})$
  for $\Gamma$.
  
  From now on, we distinguish algebraic and logical operators by adding a dot 
  on top of the algebraic ones, e.g.~$\alambda$.

We now define an interpretation of formulas into algebras.%
\footnote{Instead of the definition we give below, one can instead
use homomorphisms from the formula algebra of a language $\lan{L}$
to interpret formulas in algebras of type $\lan{L}$.}

\begin{definition}
  Let $\alg{A}$ be an algebra of type $\lan{L}=(Op,ar)$.
  A \emph{valuation} on $\alg{A}$ is a function $\val:\Prop\to\alg{A}$.
  Given a valuation $\val$ over $\alg{A}$ and a formula $\phi\in\form_\lan{L}$,
  we recursively define the \emph{interpretation} $\int{\val}(\phi)$ of $\phi$ 
  in $\alg{A}$ via $\val$ as follows\footnote{As $Form_\lan{L}$ is the free algebra in
  $\lan{L}$ over $Prop$, this is equivalent to giving a homomorphism from
$Form_\lan{L}$ to $\alg{A}$.}:
  \begin{center}
  \begin{tabular}{l c l}
  $\int{\val}(p)$ & $:=$ & $\val(p)$ \\
  $\int{\val}(\lambda\phi_1\dots\phi_n)$ & $:=$ & $\alambda(\int{\val}(\phi_1))\dots(\int{\val}(\phi_n))$ \\
  \end{tabular}
  \end{center}
\end{definition}

We turn to the definition of equational consequence, which forms the semantic side of
abstract algebraic logic.

\begin{definition}\label{def:eqcsq}
  An \emph{equation} is a pair of two formulas, usually written $\phi = \psi$. 

  For a class of algebras $\clalg{K}$, we define the \emph{equational consequence} on $\clalg{K}$ 
  as follows:
  \begin{center}
  \begin{tabular}{l c l}
  $\Theta \vDash_{K} \phi = \psi$ & $if$ &
  $\forall\alg{A}\in\clalg{K}.\,\forall\val.\;\;[\forall\theta.\forall\eta.\;(\theta=\eta)\in\Theta\;\Rightarrow\;\int{\val}(\theta)=\int{\val}(\eta)]\;\Rightarrow\;\int{\val}(\phi)=\int{\val}(\psi)$ \\
  \end{tabular}
  \end{center}
  Where $\Theta$ is a set of equations.

  An algebra $\alg{A}$ \emph{validates}
  $\phi = \psi$ if $\vDash_{\{\alg{A}\}} \phi = \psi$, 
  which we also write as $\alg{A}\vDash \phi = \psi$.
  Similarly, a class of algebras $\clalg{K}$ validates $\phi = \psi$
  if $\vDash_{\clalg{K}} \phi = \psi$ which we also write $\clalg{K}\vDash \phi = \psi$.
\end{definition}

A central notion in universal algebra is that of an algebraic variety.
While this concept only plays a minor role in our work,
we still make use of it.%
\footnote{Quasi-varieties also play an important role in \AAL~and
  universal algebra, but we choose not to introduce this notion here as the only
  algebraic class we treat in detail are bi-Heyting algebras, which are a
  variety. A few of the theorems of \AAL~have been accordingly (but slightly) rewritten to
  only apply to varieties.}

\begin{definition}
  A class of algebras $\clalg{K}$ is a \emph{variety} if there is a set of equations $\Theta$ 
  such that $\clalg{K} =\{\alg{A} \mid \alg{A} \vDash \Theta\}$.
\end{definition}

As groups, rings, Heyting algebras and Boolean algebras can be defined
equationally, they are varieties.

\subsection{Logics in the Leibniz hierarchy}

Here, we define types of logics which appear in the Leibniz hierarchy.
To do so, we first need to define in the abstract what a logic is.
The notion we use originates with Tarski~\cite{Tar30},
where he considered logics as consequence relations.
His definition was in turn extended with substitution invariance
by \L{}o\'{s} and Suzko~\cite{LosSuz1958}.
This understanding of logic as consequence relation
 is a cornerstone of \AAL.

\begin{definition}
  \label{def:logic}
  A \emph{logic} $\log{L}$ of type $\lan{L}$, is a relation, usually written $\vdash_{\log{L}}$,
  from $\mathcal{P}(\form_\lan{L})$ to $\form_\lan{L}$ such that the following hold.
  \begin{enumerate}
    \item[(I)] If $\phi \in \Gamma$, then $\Gamma \vdash_{\log{L}} \phi$;
    \item[(M)] If $\Gamma \vdash_{\log{L}} \phi$ and $\Gamma \subseteq \Delta$, then $\Delta \vdash_{\log{L}} \phi$;
    \item[(C)] If $\Gamma \vdash_{\log{L}} \phi$ and $\Delta \vdash \psi$ for every $\psi \in \Gamma$, then
      $\Delta \vdash_{\log{L}} \phi$;
    \item[(S)] If $\Gamma \vdash_{\log{L}} \phi$, then $\sigma\Gamma \vdash_{\log{L}} \sigma \phi$ 
      for every substitution $\sigma$.
  \end{enumerate}
  The conditions I,M,C,S are also referred to as Identity, Monotonicity, Cut and
  Structurality, respectively. When the context is clear, we drop the subscript in $\vdash_{\log{L}}$.

  We make use of the standard shorthands, writing
  \begin{center}
  \begin{tabular}{r c l}
    $\vdash\phi$ & for & $\emptyset \vdash \phi$ \\
    $\Gamma, \psi \vdash \phi$ & for & $\Gamma \cup \{\psi\} \vdash \phi$ \\
    $\Gamma, \Delta \vdash \phi$ & for & $\Gamma \cup \Delta \vdash \phi$ \\
    $\Gamma \vdash \Delta$ & for & $\Gamma \vdash \delta \text{ for all }
    \delta \in \Delta$ \\
    $\Gamma \dashv\vdash \Delta$ & for & $\Gamma \vdash \Delta \text{ and }
    \Delta \vdash \Gamma.$
  \end{tabular}
  \end{center}

  A \emph{theory} is a set of formulas closed under consequence, 
  that is a set $\Gamma$ such that $\Gamma \vdash \phi$ implies $\phi \in \Gamma$.
  For any logic we denote its set of theories as $Th_{\vdash}$.

  A logic is \emph{finitary} if, given $\Gamma \vdash \phi$, then there exists a finite
  $\Delta\subseteq\Gamma$ such that $\Delta \vdash \phi$.
\end{definition}

Next, relying on Definition~\ref{def:eqcsq}, we define what an algebraic semantics is
from the perspective of abstract algebraic logic.

\begin{definition}\label{def:taualgsem}
  For a logic $\log{L}$, a class of algebras $\clalg{K}$ and a set of equations in one
  variable $\tau(x)$, $\log{L}$ is said to have a \emph{$\tau(x)$-algebraic
  semantics} if for all $\Gamma \cup \{\phi\} \subseteq\form$,
  $$ 
  \Gamma \vdash \phi \iff \tau(\Gamma) \vDash_\clalg{K} \tau(\phi)
  $$
  where $\tau(\phi)$ is the set of equations $\tau(x)$ where all occurrences of $x$
  are replaced by $\phi$, and $\tau(\Gamma)=\bigcup\{\tau(\gamma)\mid\gamma\in\Gamma\}$.
  We sometimes omit the mention of $x$ in $\tau(x)$ and allow ourselves to write $\tau$.
\end{definition}

The notion of $\tau(x)$-algebraic semantics is not yet strong enough to connect
algebras and logics.
Indeed, many logics have strange and non-standard algebraic semantics.
For instance, following Glivenko's theorem~\cite{Gli29}, Heyting algebras
are a $\{\neg \neg x = \top\}$-algebraic semantics for classical logic.
Therefore, we need a way of ignoring pathological examples like this.
To do so, we introduce the notion of algebraizability.
The following definition was first given by Blok and Pigozzi~\cite{BloPig89}, with a
textbook account given by Font~\cite[Definition 3.11]{Fon16}.

\begin{definition}\label{def:algebraizable}
  A logic $\log{L}$ is \emph{algebraizable with respect to a class
  of algebras $\clalg{K}$} if there exists
  a set of equations $\tau(x)$, called \emph{defining equations},
  and formulas $\Delta(x,y)$, called \emph{equivalence formulas}, such that
  for all $\Gamma \cup \{\phi\} \subseteq\form$
  and all sets of equations $\Theta \cup \{\epsilon = \delta\}$, the
  following conditions are satisfied.
  \begin{center}
  \begin{tabular}{c @{\hspace{1cm}} l}
    (ALG1) & $\Gamma \vdash \phi \iff \tau(\Gamma) \vDash_\clalg{K} \tau(\phi)$ \\
    \rule{0pt}{2.5ex} 
    (ALG2) & $\Theta \vDash_\clalg{K} \epsilon = \delta \iff \Delta(\Theta) \vdash \Delta(\epsilon,\delta)$ \\
    \rule{0pt}{2.5ex} 
    (ALG3) & $\phi \dashv\vdash \Delta(\tau(\phi))$ \\
    \rule{0pt}{2.5ex} 
    (ALG4) & $\epsilon = \delta \Dashv\vDash_\clalg{K} \tau(\Delta(\epsilon,\delta))$
  \end{tabular}
  \end{center}
  Note that we write $\Delta(\epsilon,\delta)$ to mean the set of formulas $\Delta(x,y)$
  where all occurrences of $x$ and $y$ are respectively replaced by $\epsilon$ and $\delta$,
  and $\Delta(\Theta)=\bigcup\{\Delta(\epsilon,\delta)\mid (\epsilon=\delta)\in\Theta\}$ for a set of equations $\Theta$.

  We also say that $\log{L}$ is \emph{algebraizable} if there exists a
  class of algebras $\clalg{K}$ such that $\log{L}$ is algebraizable
  with respect to that class.
  If the defining equations and the equivalence formulas are finite, we say that
  $\log{L}$ is \emph{finitely algebraizable}.
\end{definition}

A logic can be algebraizable with respect to several classes of algebras.
However, there is a largest such class. 

\begin{definition}
  For any algebraizable logic $\log{L}$, the \emph{equivalent algebraic semantics} 
  of $\log{L}$ is the largest class of algebras such that $\log{L}$ is algebraizable 
  with respect to that class. 
\end{definition}

The existence of an equivalent algebraic semantics for every algebraizable logic \cite[Corollary 3.18]{Fon16}
ensures that the above is well-defined.

\begin{example}
  Boolean algebras are the equivalent algebraic semantics of classical logic,
  Heyting algebras are the equivalent algebraic semantics of intuitionistic logic, and
  modal algebras are the equivalent algebraic semantics of the global modal logic $\log{K}$\cite[Page 120]{Fon16}.
\end{example}

Font \cite[Corollary 3.18]{Fon16} gives an easy way to check if a variety is 
the equivalent algebraic semantics of a logic.

\begin{theorem}\label{thrm:easqv}
  If $\log{L}$ is algebraizable with respect to a variety $\clalg{V}$,
  then $\clalg{V}$ is the equivalent algebraic semantics of $\log{L}$.
\end{theorem}

Let us shift our focus to syntactic considerations on classes of logics in the
\AAL~setting.
The following definition is a syntactic characterisation of what it means for a
logic to be algebraizable~\cite[Section 3.3]{Fon16}.

\begin{theorem}\label{thm:syntaxalgebraizable}
  A logic $\log{L}$ is \emph{algebraizable} if and only if there exist 
  a set of defining equations $\tau(x)$ and equivalence formulas $\Delta(x,y)$
  such that the following conditions, together with (ALG3) from Definition~\ref{def:algebraizable}, hold.
  \begin{center}
  \begin{tabular}{r @{\hspace{1cm}} l}
    (R) & $\vdash \Delta(x,x)$ \\
    \rule{0pt}{2.5ex} 
    (Sym) & $\Delta(x,y) \vdash \Delta(y,x)$ \\
    \rule{0pt}{2.5ex} 
    (Trans) & $\Delta(x,y), \Delta(y,z) \vdash \Delta(x,z)$ \\
    \rule{0pt}{2.5ex} 
    (Re) & $\Delta(x_1,y_1), ... \Delta(x_n,y_n) \vdash 
      \Delta(\lambda x_1 ... x_n, \lambda y_1 ... y_n)$ 
      for any $\lambda\in Op$ with $ar(\lambda)=n$
  \end{tabular}
  \end{center}
\end{theorem}

The first three condition tell us that $\Delta(x,y)$ forms an equivalence relation,
while (Re) specifies it to a congruence.
Finally, (ALG3) ensures that the defining equations and equivalence formulas
do not lose information over the translation from logic to algebra and back to logic.

Most standard logics (intuitionistic, classical and global modal) are
algebraizable with equivalence formulas $\{x \to y, y \to x\}$ and defining
equation $x = \top$.
Rasiowa and Sikorski~\cite{RasSik63} identified this pattern before the creation of \AAL,
and abstracted the properties of implication necessary to be 
algebraizable in the following way.

\begin{definition}\label{def:implicative}
  An \emph{implicative logic} is some logic $\log{L}$ over a type $\lan{L}$
  with a binary operator $\to$ such that the following holds, for any $p,q,r \in\Prop$ and
  $\lambda\in Op$ with $ar(\lambda)=n$.
  \begin{center}
  \begin{tabular}{r @{\hspace{1cm}} l}
    (IL1) & $\vdash p \to p$ \\
    \rule{0pt}{2.5ex} 
    (IL2) & $p \to q, q \to r \vdash p \to r$ \\
    \rule{0pt}{2.5ex} 
    (IL3) & $\{p_k\to q_k\mid 1\leq k\leq n\} \cup \{q_k\to p_k\mid 1\leq k\leq n\} \vdash 
      \lambda p_1 ... p_n \to \lambda q_1 ... q_n$   \\
    \rule{0pt}{2.5ex} 
    (IL4) & $p, p \to q \vdash q$ \\
    \rule{0pt}{2.5ex} 
    (IL5) & $p \vdash q \to p$ \\
  \end{tabular}
  \end{center}
\end{definition}

Using the syntactic characterisation of algebraizability of Theorem~\ref{thm:syntaxalgebraizable},
it can easily be shown that implicative logics are algebraizable~\cite[Proposition 3.15]{Fon16}.
Therefore, implicative logics are sitting higher up in the Leibniz hierarchy.

\begin{proposition}
  \label{prop:imparealg}
  Implicative logics are finitely algebraizable with defining equation 
  $\{x = x \to x\}$ and equivalence formulas $\{x \to y, y \to x\}$.
\end{proposition}

Now, we use some of the properties of Theorem~\ref{thm:syntaxalgebraizable}
to syntactically define significant classes of logics sitting below algebraizable
ones in the hierarchy~\cite[Section 6.1]{Fon16}.

\begin{definition}\label{def:equivalential}
  A logic is called \emph{protoalgebraic} if it has a set of equivalence formulas that 
  validate (R) and (MP')%
  \footnote{Usually this condition is called (MP), but we add the prime here
  to distinguish it from the proof-theoretic rule \GHrule{MP} given in Figure~\ref{fig:Hilbert}.},
  shown below. 
  \begin{center}
  \begin{tabular}{r @{\hspace{1cm}} l}
    (MP') & $x,\Delta(x,y) \vdash y$ \\
  \end{tabular}
  \end{center}
  An \emph{equivalential logic} is a protoalgebraic whose
  equivalence formulas also validate (Re).
  For protoalgebraic and equivalential logics, the set $\Delta(x,y)$
  is also called the set of \emph{congruence formulas}.
  A finitely equivalential logic is an equivalential logic where $\Delta(x,y)$ is finite.
\end{definition}

To close this section, let us introduce the so-called ``Isomorphism theorem",
which presents an isomorphism between relative congruences
and \emph{deductive filters} for algebraizable logics.
We start by defining deductive filters.

\begin{definition}\label{def:logicfilters}
  For any logic $\log{L}$ and algebra $\alg{A}$, a \emph{$\log{L}$-filter} (or \emph{deductive filter}) is a set $F
  \subseteq \alg{A}$ such that, for all $\Gamma \cup \{\phi\} \in \form$ and valuation $\val$ over $\alg{A}$:
  $$
  \text{if } \Gamma \vdash \phi \text{ and } \int{\val}(\Gamma) \subseteq F \text{ then } \int{\val}(\phi) \in F.
  $$
  We let $Fi_\vdash(\alg{A})$ be the set of deductive filters on an algebra $\alg{A}$.
\end{definition}

\begin{example}
  Classical logic-filters on Boolean algebras are precisely the lattice filters,
  as are intuitionistic logic-filters on Heyting algebras.
  Global modal logic $\log{K}$-filters on modal algebras are precisely the open
  filters i.e.~the filters $F$ such that if $p \in F$ then $\square p \in F$\cite[Exercise 2.31]{Fon16}.
\end{example}

We now give the Isomorphism theorem, which was first proven by
Blok and Pigozzi~\cite{BloPig89}.

\begin{theorem}[The Isomorphism Theorem]
  \label{thrm:theisothrm}
  If $\log{L}$ is algebraizable with equivalent algebraic semantics $\clalg{K}$, 
  for any algebra $\alg{A}$ (not necessarily in $\clalg{K}$), then there is a lattice 
  isomorphism between $Con_\clalg{K}(\alg{A})$ and $Fi_\vdash(\alg{A})$.
\end{theorem}

In our paper, we exploit this theorem to reduce a proof that a logic
is not algebraizable to a proof of the absence of such an isomorphism.


\section{Basics of bi-intuitionistic logics}\label{sec:biint}

In this section we present the basics of the two bi-intuitionistic logics $\wbil$ and $\sbil$ we are analysing
through the tools of \AAL.
More precisely, we introduce their syntax, axiomatic proof systems, Kripke semantics
and known facts of relevance.
Most of what is presented here can be found in the work of Gor\'{e} and Shillito~\cite{GorShi20}
and in Shillito's PhD thesis~\cite{Shi23}.


\subsection{Syntax}

As mentioned above, bi-intuitionistic logics are expressed in the language of intuitionistic logic extended with the exclusion operator~$\excl$. More formally: 

\begin{definition}[\coqdoc{Synt.Syntax.html\#form}]
We define the propositional language $\lanbi=(Op_{BI},ar_{BI})$ where $Op_{BI}=\{\top,\bot,\land,\lor,\to,\excl\}$
and $ar_{BI}(\top)=ar_{BI}(\bot)=0$ and $ar_{BI}(\land)=ar_{BI}(\lor)=ar_{BI}(\to)=ar_{BI}(\excl)=2$. Alternatively, the formulas of $\form_{\lanbi}$ can be seen as generated by the following grammar in Backus-Naur form:
\[
   \phi ::= p\in\Prop \mid \bot \mid \top \mid \phi\,\land\,\phi \mid \phi\,\lor\,\phi \mid \phi\,\to\,\phi \mid \phi\,\excl\,\phi
\]
We define $\phi\leftrightarrow\psi:=(\phi\to\psi)\land(\psi\to\phi)$.
We also define the abbreviations $\neg\phi:=(\phi\to\bot)$ and $\wneg\phi:=(\top\excl\phi)$, respectively called \emph{negation} and \emph{co-negation} (or also \emph{weak} negation).
\end{definition}

The added binary operator $\phi\excl\psi$ is intended to be the dual of
$\phi\to\psi$ and is usually read as ``$\phi$ excludes $\psi$''. Consequently,
$\wneg$ is also the dual of $\neg$. As we often use repetitions of the pattern $\DN$, 
we recursively define $(\DN)^n\phi$~(\coqdoc{Synt.Syntax.html\#DN_form}),
for a formula $\phi$ and a natural number $n$, to be $\phi$ if $n=0$,
and $\DN(\DN)^{n-1}\phi$ otherwise.

For a formula $\phi$, we recursively define its bi-depth $d(\phi)$~(\coqdoc{Synt.Syntax.html\#depth}) on the structure of $\phi$: $d(p)=d(\bot)=d(\bot)=1$  for $p\in Prop$, $d(\phi_1\land\phi_2)= d(\phi_1\lor\phi_2)=max(d(\phi_1),d(\phi_2))$, and $d(\phi_1\to\phi_2)= d(\phi_1\excl\phi_2)=max(d(\phi_1),d(\phi_2)) +1$.

In the following, when $\Gamma$ refers to a set of formulas, we write $\Gamma,\phi$ or $\phi,\Gamma$ to mean $\Gamma\cup\{\phi\}$.


\subsection{Axiomatic proof theory}

Following \AAL, we focus on logics as consequence relations and
axiomatically capture the two bi-intuitionistic logics using \emph{generalised Hilbert calculi},
which manipulate consecutions.

The generalised Hilbert calculi for $\wbil$ (\coqdoc{GenHil.BiInt_GHC.html\#wBIH_prv}) and $\sbil$ (\coqdoc{GenHil.BiInt_GHC.html\#sBIH_prv})~\cite{GorShi20}, extend
the one for intuitionistic logic, containing the axioms $A_1$ to $A_{10}$,
with the axioms $A_{11}$ to $A_{14}$ together with the rule \GHrule{wDN} for $\wbil$
and the rule \GHrule{sDN} for $\sbil$, all shown in Figure~\ref{fig:Hilbert}.
There, $\mathcal A$ in the rule \GHrule{Ax} refers to the set of all instances of axioms.

A derivation of $\Gamma\vdash\phi$ in $\wbil$ (resp.~$\sbil$) is a tree
of consecutions built using the rules \GHrule{MP} and \GHrule{wDN}
(resp.~\GHrule{sDN}) in Figure~\ref{fig:Hilbert} with instances of \GHrule{Ax}
and \GHrule{El} as leaves.
If there is a derivation of $\Gamma\vdash\phi$ in $\wbil$ (resp.~$\sbil$),
we write $\wderiv{\Gamma}{\phi}$ (resp.~$\sderiv{\Gamma}{\phi}$).
We write $\nowderiv{\Gamma}{\phi}$ (resp.~$\nosderiv{\Gamma}{\phi}$)
if $\wderiv{\Gamma}{\phi}$ (resp.~$\sderiv{\Gamma}{\phi}$) does not hold.
We also write $\ideriv{\Gamma}{\phi}$ to mean that \emph{both} 
$\wderiv{\Gamma}{\phi}$ and $\sderiv{\Gamma}{\phi}$,
and similarly for $\noideriv{\Gamma}{\phi}$.
We formally define the logics $\wbil$ and $\sbil$ as the sets 
$\{(\Gamma,\phi):\;\wderiv{\Gamma}{\phi}\}$ and
$\{(\Gamma,\phi):\;\sderiv{\Gamma}{\phi}\}$ respectively.

\begin{figure}[t]
\begin{center}
\small
\begin{tabular}{l@{\hspace{0.7cm}}l @{\hspace{1cm}} l@{\hspace{0.7cm}}l}
$A_{1}$ & $\phi\to(\psi\to\phi)$ & $A_{2}$ & $(\phi\to(\psi\to\chi))\to((\phi\to\psi)\to(\phi\to\chi))$ \\
$A_{3}$ & $\phi\to(\phi\lor\psi)$ & $A_{4}$ & $\psi\to(\phi\lor\psi)$ \\
$A_{5}$ & $(\phi\to\chi)\to((\psi\to\chi)\to((\phi\lor\psi)\to\chi))$ & $A_{6}$ & $(\phi\land\psi)\to\phi$ \\
$A_{7}$ & $(\phi\land\psi)\to\psi$ & $A_{8}$ & $(\chi\to\phi)\to((\chi\to\psi)\to(\chi\to(\phi\land\psi)))$ \\
$A_{9}$ & $\bot\to\phi$ & $A_{10}$ & $\phi\to\top$ \\
 & & & \\
$A_{11}$ & $\phi\to(\psi\lor(\phi\excl\psi))$ & $A_{12}$ & $(\phi\excl\psi)\to\wneg(\phi\to\psi)$ \\
$A_{13}$ & $((\phi\excl\psi)\excl\chi)\to(\phi\excl(\psi\lor\chi))$ & $A_{14}$ & $\neg(\phi\excl\psi)\to(\phi\to\psi)$ \\
\end{tabular}
\end{center}

\begin{center}
\begin{tabular}{c@{\hspace{1cm}}c@{\hspace{0.8cm}}c@{\hspace{0.8cm}}c@{\hspace{0.8cm}}c}
$
\inferLineSkip=3pt
\infer[\scriptstyle\GHrule{Ax}]{\Gamma\vdash\phi}{\phi\in\mathcal{A}}
$ & 
$
\inferLineSkip=3pt
\infer[\scriptstyle\GHrule{El}]{\Gamma\vdash\phi}{\phi\in\Gamma}
$ &
$
\inferLineSkip=3pt
\infer[\scriptstyle\GHrule{MP}]{\Gamma\vdash\psi}{
	\Gamma\vdash\phi
	&
	\Gamma\vdash\phi\to\psi}
$ &
$
\inferLineSkip=3pt
\infer[\scriptstyle \GHrule{wDN}]{\Gamma\vdash\DN\phi}{\emptyset\vdash\phi}
$ &
$
\inferLineSkip=3pt
\infer[\scriptstyle \GHrule{sDN}]{\Gamma\vdash\DN\phi}{\Gamma\vdash\phi}
$\\
\end{tabular}
\end{center}

\caption{Generalised Hilbert calculus $\wbil$ for $\wbil$}
\label{fig:Hilbert}
\end{figure}

Here, we present constructively-obtained proof-theoretic results from the work of Shillito \cite{Shi23}.
First, Shillito proves that both logics are \emph{finitary logics} (\coqdoc{GenHil.BiInt_logics.html}) (see Definition~\ref{def:logic}).
Interesting properties of these logics, sometimes exclusive to one, are shown below.

\begin{lemma}
\label{lem:derivations}
We have the following:
\begin{center}
\begin{tabular}{r @{\hspace{0.3cm}} l @{\hspace{0.5cm}} c @{\hspace{0.5cm}} l @{\hspace{0.3cm}} c}
$1.$ & $\ideriv{\Gamma}{\phi\lor\wneg\phi}$ &  &  & 
$(\coqdoc{GenHil.wBIH_meta_interactions.html\#BiLEM})$\\
$2.$ & $\ideriv{\emptyset}{(\phi\excl\psi)\to\chi}$ & $\Leftrightarrow$ & $\ideriv{\emptyset}{\phi\to(\psi\lor\chi)}$ & 
$(\coqdoc{GenHil.wBIH_meta_interactions.html\#dual_residuation})$\\
$3.$ & $\sderiv{\Gamma}{(\phi\excl\psi)\to\chi}$ & $\Leftrightarrow$ & $\sderiv{\Gamma}{\phi\to(\psi\lor\chi)}$ & 
$(\coqdoc{GenHil.sBIH_meta_interactions.html\#sdual_residuation})$\\
$4.$ & $\wderiv{\Gamma,\phi}{\psi}$ & $\Leftrightarrow$ & $\wderiv{\Gamma}{\phi\to\psi}$ & 
$(\coqdoc{GenHil.wBIH_meta_interactions.html\#wBIH_Deduction_Theorem},
\coqdoc{GenHil.wBIH_meta_interactions.html\#wBIH_Detachment_Theorem})$\\
$5.$ & $\sderiv{\Gamma,\phi}{\psi}$ & $\Leftrightarrow$ & $\exists n\in\N.\;\;\sderiv{\Gamma}{(\DN)^n\phi\to\psi}$ & 
$(\coqdoc{GenHil.sBIH_meta_interactions.html\#gen_sBIH_Deduction_Theorem},\coqdoc{GenHil.sBIH_meta_interactions.html\#gen_sBIH_Double_Negated_Detachment_Theorem})$\\
\end{tabular}
\end{center}
\end{lemma}

(1) above shows that a bi-intuitionistic version of the law of excluded-middle holds in both logics.
(2) is a syntactic analogue of the algebraic dual residuation property below.
$$
\inferLineSkip=3pt
\infer={a\excl b\leq c}{a\leq b\lor c}
$$
This law holds in $\sbil$ even when the context of the consecution is not empty, as shown in (3).
While (4) is the deduction theorem for $\wbil$,
it does not hold for $\sbil$, which satisfies the modified version in (5).


\subsection{Relational semantics}

We proceed to define the traditional relational (Kripke) semantics for the bi-intuitionistic language~\cite{Rau80}.
This semantics uses (Kripke) models identical to the ones of intuitionistic logic.

\begin{definition}[\coqdoc{Krip.BiInt_Kripke_sem.html\#model}]
A model $\mo{M}$ is a tuple $(W,\leq,I)$,
where $(W,\leq)$ is a preordered set and
$I:\Prop\to\mathcal P(W)$ is a \emph{persistent} interpretation function
(where $\mathcal P(W)$ is the powerset of $W$) i.e: 
$$\forall v,w\in W.\,\forall p\in\Prop.\;(w\leq v\,\Rightarrow\, w\in I(p)\,\Rightarrow\, v\in I(p))$$
\end{definition}

This semantics extends the forcing relation of intuitionistic logic to $\excl$ in the following way. 

\begin{definition}[\coqdoc{Krip.BiInt_Kripke_sem.html\#wforces}]\label{DefForcBiInt}
Given a model $\mo{M}=(W,\leq,I)$, we define the forcing relation $\mo{M},w\Vdash \phi$
between a world $w\in W$ and a formula $\phi$ as follows, where $\mo{M},w\not\Vdash \phi$
means that $\mo{M},w\Vdash \phi$ does not hold:
\begin{center}
\begin{tabular}{l@{\hspace{0.2cm}}c@{\hspace{0.2cm}}l}
$\mo{M},w\Vdash p$ & $:=$ & $w\in I(p)$\\
$\mo{M},w\Vdash\top$ & $:=$ & always\\
$\mo{M},w\Vdash\bot$ & $:=$ & never\\
$\mo{M},w\Vdash\phi\land\psi$ & $:=$ & $\mo{M},w\Vdash\phi\text{ and }\mo{M},w\Vdash\psi$\\
$\mo{M},w\Vdash\phi\lor\psi$ & $:=$ & $\mo{M},w\Vdash\phi\text{ or }\mo{M},w\Vdash\psi$\\
$\mo{M},w\Vdash\phi\to\psi$ & $:=$ & $\forall v\geq w.\;(\mo{M},v\Vdash\phi\text{ implies }\;\mo{M},v\Vdash\psi)$\\
$\mo{M},w\Vdash\phi\excl\psi$ & $:=$ & $\exists v\leq w.\;(\mo{M},v\Vdash\phi\text{ and }\mo{M},v\not\Vdash\psi)$\\
\end{tabular}
\end{center}
\end{definition}

Crucially, while the semantic clause for $\to$ looks \emph{forward} on the
relation $\leq$, the clause for $\excl$ looks \emph{backward}.%
\footnote{Note that in our formalisation, we use the classically equivalent but intuitionistically
weaker interpretation of $\excl$ given by Shillito and Kirst~\cite{ShiKir24}:
$(\forall v\leq w.(\mo{M},v\Vdash\phi\text{ implies }\mo{M},v\Vdash\psi))$ is false.}
This causes bi-intuitionistic logic to share similarities with tense logic~\cite{Pri55,Pri67,Pri68}.

As already mentioned, the $\DN$ pattern plays a role of importance in our work, as
it is crucial in the distinction between $\wbil$ and $\sbil$ via the rules
$\GHrule{wDN}$ and $\GHrule{sDN}$.
Here, we would like to draw the attention of the reader to the semantic
meaning of this pattern.
To see it, we introduce the following notions.

\begin{definition}
  \label{def:zigzag}
  Given a Kripke model $\mo{M}$, and two elements $w,v \in W$
  we write $w \zz v$ if there exists a $u \in A$ such that $w \leq u$ and $v \leq u$.
  We call the relation $\zz$ the \emph{zig-zag} relation. 
  Additionally, we write $w \zz^n v$ if there exists
  $u_1,...,u_{n-1}$ such that $w \zz u_1 \zz u_2 \dots u_{n-1} \zz v$, 
  with special case $w \zz^0 v$ when $w = v$~(\coqdoc{Krip.sBIH_completeness.html\#n_zz}). 
  We say $w$ and $v$ are \emph{connected}, written $v \bowtie w$, 
  if $w \zz^n v$ for some $n$~(\coqdoc{Krip.sBIH_completeness.html\#zz}). 
  \end{definition}
  
  One can view $\DN$ as a $\Box$ modality over the relation $\zz$~\cite[Lemma 8.9.2]{Shi23}.
  
  \begin{lemma}[\coqdoc{Alg.wBIH_equivalential.html\#n_zz_DN_clos_equiv}]
  \label{lem:jellyfish}
  For any Kripke model $\mo{M}$ and world $w$ in $\mo{M}$, the following holds.
  \begin{center}
  \begin{tabular}{l c l}
  $\mo{M}, w \Vdash (\DN)^n\phi$ & $\Leftrightarrow$ & $\forall v.\;w\zz^n v \; \Rightarrow \; \mo{M}, v \Vdash \phi$
  \end{tabular}
  \end{center}
\end{lemma}

\begin{proof}
  First note that $\DN\phi$ is just a shorthand for $(\top \excl \phi) \to \bot$. 
  So, $\mo{M}, w \Vdash (\top \excl \phi) \to \bot$ entails that for all $u \geq w$ we have 
  that $\mo{M}, u \Vdash (\top \excl \phi)$ implies $\mo{M}, u \Vdash \bot$.
  Given that $\mo{M}, u \Vdash \bot$ is a contradiction, we know that $\mo{M}, u \not \Vdash (\top \excl \phi)$
  holds for all $u \geq w$.
  This means that there cannot exist a world $v \leq u$ such that $\mo{M}, v \not \Vdash \phi$.
  Hence if $\mo{M}, v \Vdash \DN\phi$, then $\mo{M}, v \Vdash \phi$ for any world $v$ such that $v
  \zz w$.

  By our previous observation, and an induction on the natural numbers, we can conclude that 
  $\mo{M}, w \Vdash (\DN)^n \phi$ if and only if for all $v \zz^n w$ we have $\mo{M}, v \Vdash \phi$.
\end{proof}

We conclude our digression on $\DN$ by a way to visualise what $\DN\phi$ means on a model: 
it is as a sort of \emph{jellyfish}.
The forcing of this formula in a world $w$ entails the forcing of $\phi$
in the downset of all the elements in the upset of $w$.
In this light, the upset of $w$ is the body of the jellyfish with the
downset of the upsets its tentacles.

An important feature of the relational semantics for intuitionistic logic is
persistence. The second author~\cite[Lemma 8.7.1]{Shi23} proves that persistence is
preserved in the semantics extended to $\excl$.

\begin{lemma}[Persistence \coqdoc{Krip.BiInt_Kripke_sem.html\#Persistence}]\label{BILPers}
Let $\mo{M}=(W,\leq,I)$ be a model. The following holds.
$$\forall v,w\in W.\,\forall\phi.\,(w\leq v\,\Rightarrow\, (\mo{M},w\Vdash\phi\,\Rightarrow\, \mo{M},v\Vdash\phi))$$
\end{lemma}

The models in our semantics can be more or less similar, depending on the notion of ``similarity" we have in mind.
Such a notion, which we exploit in this paper, is \emph{$n$-bisimilarity}.
We take inspiration from Badia~\cite[Definition 1]{Bad16} and
Pattinson and de Groot \cite[Definition 4.1]{GroPat19} to define $n$-bisimulation;
although our definition is equivalent to neither.%
\footnote{To be precise, we take the idea to define 
$n$-bisimulations from Badia, and we take the
general form of bi-intuitionistic bisimulation from Pattinson 
and de Groot.}

\begin{definition}[$n$-bisimilarity \coqdoc{Krip.BiInt_bisimulation.html\#n_bisimulation}]\label{def:nbisim}
Let $\mo{M}=(W,\leq,I)$ and $\mo{M}'=(W',\leq',I')$ be models, $w$ be a world in
$W$ and $w'$ be a world in $W'$.
We define the property $\mo{M},w \bisim_n \mo{M}',w'$ recursively on $n\in\mathbb N$.
We have $\mo{M},w \bisim_0 \mo{M}',w'$ if the condition (B1) below is satisfied.
\begin{center}
\begin{tabular}{l l}
(B1) & $\forall p\in Prop.\,(w\in I(p) \; \Leftrightarrow \; w'\in I'(p))$ \\
\end{tabular}
\end{center}
We have $\mo{M},w \bisim_{m+1} \mo{M}',w'$ if the conditions (B1-B5), are satisfied.
\begin{center}
\begin{tabular}{l r l c l r}
(B2) & $\forall v\in W.\,($ & $v\leq w$ & $\Rightarrow$ & $\exists v'\in W'.\,(v'\leq' w'\;\text{  and  }\;\mo{M},v \bisim_m \mo{M}',v')$ & $)$ \\
(B3) & $\forall v\in W.\,($ & $w\leq v$ & $\Rightarrow$ & $\exists v'\in W'.\,(w'\leq' v'\;\text{  and  }\;\mo{M},v \bisim_m \mo{M}',v')$ & $)$ \\
(B4) & $\forall v'\in W'.\,($ & $v'\leq' w'$ & $\Rightarrow$ & $\exists v\in W.\,(v\leq w\;\text{  and  }\;\mo{M},v \bisim_m \mo{M}',v')$ & $)$ \\
(B5) & $\forall v'\in W'.\,($ & $w'\leq' v'$ & $\Rightarrow$ & $\exists v\in W.\,(w\leq v\;\text{  and  }\;\mo{M},v \bisim_m \mo{M}',v')$ & $)$ \\
\end{tabular}
\end{center}
If $\mo{M},w \bisim_n \mo{M}',w'$ holds, we say that $w$ and $w'$ are \emph{$n$-bisimilar}.
\end{definition}

The point of this definition is that
when $w$ and $w'$ are $n$-bisimilar they are indistinguishable by formulas of
bi-depth smaller or equal to $n$. We prove this as follows.

\begin{lemma}[\coqdoc{Krip.BiInt_bisimulation.html\#n_bisimulation_imp_n_depth_equiv}]\label{lem:depthnbisim}
Let $n\in\N$, and $\mo{M}=(W,\leq,I)$ and $\mo{M}'=(W',\leq',I')$ be models.
For any $w\in W$ and $w'\in W'$, such that $\mo{M},w \bisim_n \mo{M}',w'$ we have the following:
$$\forall\phi.\,(d(\phi)\leq n\; \Rightarrow\; (\mo{M},w\Vdash\phi\,\Leftrightarrow\,\mo{M}',w'\Vdash\phi)).$$
\end{lemma}

\begin{proof}
The proof goes by primary induction on $n$ and secondary induction on the structure of $\phi$.
Note that the case where $n=0$ is trivial, as there is no formula $\phi$ with $d(\phi)\leq 0$.
So, consider the case where $n = m + 1$ for some $m\in\N$.
We now inspect the structure of $\phi$.
The base case goes through trivially thanks to conditions (B1).
We only consider the remaining cases of $\to$ and $\excl$ as the
others are straightforward. 
\begin{itemize}
\item $\phi=\psi\to\chi:$ Assume that $\mo{M},w\Vdash\psi\to\chi$. We need to show that 
	$\mo{M}',w'\Vdash\psi\to\chi$. Let $v'\in W'$ such that $w'\leq'v'$. We assume that 
	$\mo{M}',v'\Vdash\psi$ and aim to prove that $\mo{M}',v'\Vdash\psi$. Now, condition (B5) together
	with both $\mo{M},w \bisim_{m+1} \mo{M}',w'$ and $w'\leq'v'$ gives us 
	the existence of a $v\in W$ such that $w\leq v$ and $\mo{M},v \bisim_{m} \mo{M}',v'$. 
	Thus, we can use the primary induction on $\mo{M}',v'\Vdash\psi$ to get $\mo{M},v\Vdash\psi$. 
	Therefore, we get $\mo{M},v\Vdash\chi$ as $w\leq v$ and $\mo{M},w\Vdash\psi\to\chi$. 
	By primary induction again, we obtain $\mo{M}',v'\Vdash\chi$. As $v'$ is arbitrary, we get that 
	$\mo{M}',w'\Vdash\psi\to\chi$. The other direction is symmetric.
\item $\phi=\psi\excl\chi:$ Assume that $\mo{M},w\Vdash\psi\excl\chi$. By definition, there is a $v\leq w$
	such that $\mo{M},v\Vdash\psi$ and $\mo{M},v\not\Vdash\psi$. Now, condition (B2) together
	with both $\mo{M},w \bisim_{m+1} \mo{M}',w'$ and $v\leq w$ gives us 
	the existence of a $v'\in W'$ such that $v'\leq' v'$ and $\mo{M},v \bisim_{m} \mo{M}',v'$. 
	Given that $\mo{M},v\Vdash\psi$ and $\mo{M},v\not\Vdash\psi$, we obtain 
	$\mo{M}',v'\Vdash\psi$ and $\mo{M}',v†\not\Vdash\psi$ using the primary induction hypothesis. 
	This implies $\mo{M}',w'\Vdash\psi\excl\chi$, as required. Here again, the other direction is symmetric.\qedhere
\end{itemize}
\end{proof}

Finally, we define the local and global consequence relations on this relational semantics, where $\mo{M},w\Vdash\Gamma$ means $\forall\gamma\in\Gamma.(\mo{M},w\Vdash\gamma)$.
\begin{center}
\begin{tabular}{l@{\hspace{0.5cm}}c@{\hspace{0.5cm}}l}
$\Gamma\locmodels\varphi$ & if & $\forall\mo{M}.\forall w.\;(\mo{M},w\Vdash\Gamma\hspace{0.2cm}\Rightarrow\hspace{0.2cm}\mo{M},w\Vdash\varphi)$\\
$\Gamma\globmodels\varphi$ & if & $\forall\mo{M}.\;((\forall w.\,\mo{M},w\Vdash\Gamma)\hspace{0.2cm}\Rightarrow\hspace{0.2cm}(\forall w.\,\mo{M},w\Vdash\varphi))$\\
\end{tabular}
\end{center}
We write $\Gamma\nolocmodels\phi$ (resp.~$\Gamma\noglobmodels\phi$) whenever
$\Gamma\locmodels\phi$ (resp.~$\Gamma\globmodels\phi$) does not hold.


\subsection{Soundness and completeness}

Exactly as in the case of modal logic, where one distinguishes between
a local and a global modal logic~\cite{HakNeg12},
$\wbil$ and $\sbil$ correspond respectively to the
local and global consequence relations~\cite{GorShi20,Shi23}.

\begin{theorem}[cf. Theorems 8.8.1 and 8.9.1~\cite{Shi23}]\label{thm:soundcomplkripke}
The following holds.
\begin{center}
\begin{tabular}{l c l c}
$\wderiv{\Gamma}{\varphi}$ & $\Leftrightarrow$ & $\Gamma\locmodels\varphi$ & 
$(\coqdoc{Krip.wBIH_completeness.html\#wSoundCompl})$\\
$\sderiv{\Gamma}{\varphi}$ & $\Leftrightarrow$ & $\Gamma\globmodels\varphi$ & 
$(\coqdoc{Krip.sBIH_completeness.html\#sSoundCompl})$
\end{tabular}
\end{center}
\end{theorem}

Soundness is proven as usual by showing that all axioms are valid and that rules of each logic preserve the corresponding semantic consequence. 
Completeness for $\wbil$ is proven via a canonical model construction, which has prime theories as points~\cite{ShiKir24}. The proof for $\sbil$~\cite[Theorem 8.9.1]{Shi23} relies on a transformation of the canonical model for $\wbil$~(\coqdoc{Krip.sBIH_completeness.html\#pruned_M}).


\section{$\wbil$ and $\sbil$ in the Leibniz hierarchy}\label{sec:leibhier}

Relying on the semantic and proof-theoretic characterisations of $\wbil$ and $\sbil$ given in the last section, 
we can use the tools of Section~\ref{sec:aal} to situate these logics in the Leibniz hierarchy.
In this section we introduce bi-Heyting algebras, as the class $\biha$ of these algebras plays
a crucial role for both $\wbil$ and $\sbil$.
Then, we locate $\sbil$ in the hierarchy by showing that it is implicative and finitely algebraizable 
with respect to $\biha$.
Finally, we consider $\wbil$ and show that is not algebraizable but 
instead equivalential, despite not being finitely so.
The disconnection between $\wbil$ and $\biha$ in abstract algebraic logic is further 
witnessed by the absence of algebraic semantics for the former on the latter.

\subsection{Bi-Heyting algebras}

We introduce bi-Heyting algebras as extensions of Heyting algebras, which themselves extend 
bounded lattices. These notions are successively defined below.

\begin{definition}\label{def:boundlat}
A \emph{bounded lattice} is an algebraic structure $(X,\atop,\abot,\aland,\alor)$ where $X$
is a set with $\atop$ and $\abot$ as nullary operators, and $\aland$ and $\alor$ as binary
operators on $X$, satisfying the following equalities.
\begin{center}
\begin{tabular}{l c c c}
(P1) & $a \aland \atop$ & $=$ & $a$ \\
(P2) & $a \alor \abot$ & $=$ & $a$ \\
(P3) & $a \aland b$ & $=$ & $b \aland a$ \\
(P4) & $a \aland (b \aland c)$ & $=$ & $(a \aland b) \aland c$ \\
(P5) & $a \aland (a \alor b)$ & $=$ & $a$ \\
(P6) & $a \alor b$ & $=$ & $b \alor a$ \\
(P7) & $a \alor (b \alor c)$ & $=$ & $(a \alor b) \alor c$ \\
(P8) & $a \alor (a \aland b)$ & $=$ & $a$ \\
\end{tabular}
\end{center}
The binary relation $\leq$ on a bounded lattice $L$, defined as $a\leq b$ if $a = a\aland b$, 
is a partial order (reflexive \coqdoc{Alg.Bi_Heyting_Algebras.html\#aleq_refl}, transitive \coqdoc{Alg.Bi_Heyting_Algebras.html\#aleq_trans}, antisymmetric \coqdoc{Alg.Bi_Heyting_Algebras.html\#aleq_antisym}) with 
$\atop$ and $\abot$ as top (\coqdoc{Alg.Bi_Heyting_Algebras.html\#high_one}) and bottom (\coqdoc{Alg.Bi_Heyting_Algebras.html\#low_zero}) elements.
\end{definition}

While we use the order $\leq$ given above to define more naturally (bi-)Heyting algebras,
note that this is just a notation: our algebras are ultimately defined algebraically (i.e.~with
equations) and not order-theoretically (i.e.~with $\leq$ as primitive).

\begin{definition}\label{def:ha-biha}
A \emph{Heyting algebra} is a tuple $(X,\atop,\abot,\aland,\alor,\ato)$ such that 
$(X,\atop,\abot,\aland,\alor)$ is a bounded lattice and $\ato$ is a binary operator 
on $X$ satisfying the following law. 
\begin{center}
\begin{tabular}{l c c c}
$(ResI)$ & $a\aland b \leq c$ & $\Leftrightarrow$ & $a \leq b \ato c$ \\
\end{tabular}
\end{center}

A \emph{bi-Heyting algebra}~(\coqdoc{Alg.Bi_Heyting_Algebras.html\#biHalg}) is a tuple $(X,\atop,\abot,\aland,\alor,\ato,\aexcl)$ 
such that $(X,\atop,\abot,\aland,\alor,\ato)$ is a Heyting algebra and $\aexcl$ 
is a binary operator on $X$ satisfying the following law. 
\begin{center}
\begin{tabular}{l c c c}
$(ResE)$ & $a \leq b \alor c$ & $\Leftrightarrow$ & $a \aexcl b \leq c$ \\
\end{tabular}
\end{center}
\end{definition}

Bi-Heyting algebras inherit many properties of Heyting algebras, but enjoy
many more. 
Some of these additional properties which are useful to our work are 
shown below.

\begin{lemma}
  \label{lem:bi-props}
  For any bi-Heyting algebra $\alg{A}$, and $a,b \in\alg{A}$, the following hold.
  \begin{center}
  \begin{tabular}{r c @{\hspace{2.5cm}} r c @{\hspace{2.5cm}} r c}
  $1.$ & $\awneg \awneg a \leq a$ & $2.$ & $a \leq b \alor (a \aexcl b)$ & $3.$ & $a \aexcl b \leq a \aland \awneg b$
  \end{tabular}
  \end{center}
\end{lemma}
\begin{proof}
  Recall that $\awneg a$ is an abbreviation for $\atop \aexcl a$. The needed proofs
  are as follows:
  \begin{center}
  \begin{tabular}{r r c l l}
  1. & $\atop \aexcl a$ & $\leq$ & $\atop \aexcl a$ & \\
      & $\atop$ & $\leq$ & $(\atop \aexcl a) \alor a$ & (ResE) \\
      & $\atop\aexcl(\atop \aexcl a)$ & $\leq$ & $a$ & (ResE) \\
      & & & & \\
   2. & $a \aexcl b$ & $\leq$ & $a \aexcl b$ & \\
      & $a$ & $\leq$ & $b \alor (a \aexcl b)$ & (ResE) \\
      & & & & \\
  3. & $a$ & $\leq$ & $b \alor a$ & \\
      & $a$ & $\leq$ & $(b \alor a) \aland \atop$ & (boundedness of $\alg{A}$) \\
      & $a$ & $\leq$ & $(b \alor a) \aland (b \alor (\atop \aexcl b))$ & (Lemma~\ref{lem:bi-props}.2) \\
      & $a$ & $\leq$ & $b \alor (a \aland (\atop \aexcl b))$ & (distributivity of $\alor$) \\
      & $a \aexcl b$ & $\leq$ & $a \aland (\atop \aexcl b)$ & (ResE) \\
  \end{tabular}
  \end{center}
\end{proof}


\subsection{$\sbil$ is algebraizable}

We begin our study on the \AAL~perspective on bi-intuitionistic logics with a look at $\sbil$. 
In particular we seek to show that $\sbil$ is algebraizable with $\biha$ as its equivalent algebraic semantics.
First, we proceed to show that it is (finitely) algebraizable by showing that it is implicative.
Second, we refine this statement and show in a direct way that $\biha$ is 
the equivalent algebraic semantics of $\sbil$. 

Let us show that $\sbil$ is implicative in a purely proof-theoretic way.
As $\sbil$ is an extension of intuitionistic logic, the only difficulty resides
in showing (IL3) of Definition~\ref{def:implicative}. 
More precisely, when $\lambda$ is $\excl$ we cannot rely on the proof
of implicativity of intuitionistic logic~\cite{RasSik63}.

\begin{theorem}
  \label{prop:strongimp}
  The logic $\sbil$ is implicative.
\end{theorem}
\begin{proof}
  We omit the proofs for (IL1)~(\coqdoc{Alg.sBIH_implicative.html\#sBIH_IL1}), (IL2)~(\coqdoc{Alg.sBIH_implicative.html\#sBIH_IL2}), (IL4)~(\coqdoc{Alg.sBIH_implicative.html\#sBIH_IL4}) and (IL5)~(\coqdoc{Alg.sBIH_implicative.html\#sBIH_IL5})
  as they are standard proofs only relying on intuitionistically valid steps. 
  Therefore, we focus on the property (IL3).
  
  We prove the property by cases on the identity of the operator $\lambda$. 
  Here again, if $\lambda$ is $\top$, $\bot$, $\land$,
  $\lor$ or $\to$, we resort to solely intuitionistic arguments~(\coqdoc{Alg.sBIH_implicative.html\#sBIH_IL3_Top},\coqdoc{Alg.sBIH_implicative.html\#sBIH_IL3_Bot},\coqdoc{Alg.sBIH_implicative.html\#sBIH_IL3_And},\coqdoc{Alg.sBIH_implicative.html\#sBIH_IL3_Or},\coqdoc{Alg.sBIH_implicative.html\#sBIH_IL3_Imp}).
  We are thus left with providing the details for the case where $\lambda$ is $\excl$~(\coqdoc{Alg.sBIH_implicative.html\#sBIH_IL3_Excl}),
  i.e.~we need to prove that $\sderiv{p_1\to q_1,q_1\to p_1,p_2\to q_2,q_2\to p_2}{(p_1\excl q_1)\to(p_2\excl q_2)}$.
  Consider the following, where $\Gamma:=p_1\to q_1,q_1\to p_1,p_2\to q_2,q_2\to p_2$ and
  where dashed lines are applications of admissible rules (if all their premises are provable, then so is their conclusion).
  
\begin{center}
\AxiomC{}
\RightLabel{$\GHrule{El}$}
\UnaryInfC{$\Gamma\vdash q_2\to q_1$}
\RightLabel{Right Anti-Monot.}\dashedLine
\UnaryInfC{$\Gamma\vdash(p_1\excl q_1)\to(p_1\excl q_2)$}
\AxiomC{}
\RightLabel{$\GHrule{El}$}
\UnaryInfC{$\Gamma\vdash p_1\to p_2$}
\RightLabel{Left Monot.}\dashedLine
\UnaryInfC{$\Gamma\vdash(p_1\excl q_2)\to(p_2\excl q_2)$}
\RightLabel{Imp.~Trans.}\dashedLine
\BinaryInfC{$\Gamma\vdash(p_1\excl q_1)\to(p_2\excl q_2)$}
\DisplayProof
\end{center}
As shown above, our case is straightforwardly proved via
transitivity of implication~(\coqdoc{GenHil.sBIH_meta_interactions.html\#meta_sImp_trans}), together with the fact that the exclusion
is monotone on the left~(\coqdoc{GenHil.sBIH_meta_interactions.html\#smonoR_Excl}) and anti-monotone on the right~(\coqdoc{GenHil.sBIH_meta_interactions.html\#smonoL_Excl})
\emph{in the presence of a context $\Gamma$}.
\end{proof}

As a consequence of $\sbil$ being implicative and
Proposition~\ref{prop:imparealg} we have the following.

\begin{corollary}
  $\sbil$ is finitely algebraizable with defining equation $p = \top$
  and equivalence formulas $p \leftrightarrow q$.%
  \footnote{For conciseness of notation we use the formula $p\leftrightarrow q$
  instead of $\{p \to q, q \to p\}$. This is legitimate here as in our logics we have
  $p \leftrightarrow q\dashv\vdash\{p \to q, q \to p\}$.}
\end{corollary}

Next, we want to refine the statement above and exhibit a specific class of
algebras with respect to which $\sbil$ is algebraizable.
Concretely, we show that $\sbil$ is algebraizable with respect to $\biha$
with defining equation $p = \top$ and equivalence formula $p \leftrightarrow q$.
 To obtain this result, we need to show the holding of all four conditions of 
Definition~\ref{def:algebraizable}.
We first prove (ALG3) and (ALG4) in a proof-theoretic
and algebraic way, respectively.

\begin{proposition}[ALG3~\coqdoc{Alg.sBIH_algebraizable.html\#sBIH_Alg3}]\label{prop:alg3}
The following holds.
\begin{center}
\begin{tabular}{c @{\hspace{2cm}} c}
$\sderiv{\phi}{\phi\leftrightarrow\top}$ & $\sderiv{\phi\leftrightarrow\top}{\phi}$
\end{tabular}
\end{center}
\end{proposition}

\begin{proof}
Consider the proofs below, and note that the two topmost ones establish $\sderiv{\phi}{\phi\leftrightarrow\top}$.
\begin{center}
\begin{tabular}{c @{\hspace{1.4cm}} c @{\hspace{1.4cm}}}
\AxiomC{}
\RightLabel{Lem.~\coqdoc{GenHil.sBIH_meta_interactions.html\#sprv_Top}}\dashedLine
\UnaryInfC{$\phi\vdash\top$}
\AxiomC{}
\RightLabel{\GHrule{Ax}}
\UnaryInfC{$\phi\vdash\top\to\phi\to\top$}
\RightLabel{\GHrule{MP}}
\BinaryInfC{$\phi\vdash\phi\to\top$}
\DisplayProof
& 
\AxiomC{}
\RightLabel{\GHrule{El}}
\UnaryInfC{$\phi\vdash\phi$}
\AxiomC{}
\RightLabel{\GHrule{Ax}}
\UnaryInfC{$\phi\vdash\phi\to\top\to\phi$}
\RightLabel{\GHrule{MP}}
\BinaryInfC{$\phi\vdash\top\to\phi$}
\DisplayProof
\\
& \\
\multicolumn{2}{c}{
\AxiomC{}
\RightLabel{Lem.~\coqdoc{GenHil.sBIH_meta_interactions.html\#sprv_Top}}\dashedLine
\UnaryInfC{$\phi\leftrightarrow\top\vdash\top$}
\AxiomC{}
\RightLabel{\GHrule{El}}
\UnaryInfC{$\phi\leftrightarrow\top\vdash\phi\leftrightarrow\top$}
\AxiomC{}
\RightLabel{\GHrule{Ax}}
\UnaryInfC{$\phi\leftrightarrow\top\vdash(\phi\leftrightarrow\top)\to(\top\to\phi)$}
\RightLabel{\GHrule{MP}}
\BinaryInfC{$\phi\leftrightarrow\top\vdash\top\to\phi$}
\RightLabel{\GHrule{MP}}
\BinaryInfC{$\phi\leftrightarrow\top\vdash\phi$}
\DisplayProof}
\end{tabular}
\end{center}
\end{proof}

\begin{proposition}[ALG4~\coqdoc{Alg.sBIH_algebraizable.html\#sBIH_Alg4}]\label{prop:alg4}
The following holds.
\begin{center}
\begin{tabular}{l @{\hspace{1.4cm}} c}
$\epsilon=\delta\vDash_{\biha}(\epsilon\leftrightarrow\delta) = \top$ & $(\epsilon\leftrightarrow\delta) = \top\vDash_{\biha}\epsilon = \delta$
\end{tabular}
\end{center}
\end{proposition}

\begin{proof}
To prove $\epsilon=\delta\vDash_{\biha}(\epsilon\leftrightarrow\delta) = \top$ it suffices
to prove $\epsilon=\delta\vDash_{\biha}(\epsilon\to\delta) = \top$ and 
$\epsilon=\delta\vDash_{\biha}(\delta\to\epsilon) = \top$.
We provide the proof of the former but omit the one of the latter, as it is symmetric.
Let $\alg{A}\in\biha$ and $\val$ a valuation over $\alg{A}$. 
Assume that $\int{\val}(\epsilon)=\int{\val}(\delta)$. 
We need to show $\int{\val}(\epsilon)\ato\int{\val}(\delta)=\atop$. Under our assumption it
suffices to show $\int{\val}(\epsilon)\ato\int{\val}(\epsilon)=\atop$, which is straightforwardly provable.

We are left to prove $(\epsilon\leftrightarrow\delta) = \top\vDash_{\biha}\epsilon = \delta$.
Let $\alg{A}\in\biha$ and $\val$ a valuation over $\alg{A}$. 
Assume that $\int{\val}(\epsilon)\aleftrightarrow\int{\val}(\delta)=\atop$.
 By the antisymmetry of $\leq$, $\int{\val}(\epsilon)=\int{\val}(\delta)$ can be proved by
showing $\int{\val}(\epsilon)\leq\int{\val}(\delta)$ and  $\int{\val}(\epsilon)\geq\int{\val}(\delta)$.
In turn, these are easily obtained via our assumption, $(ResI)$, (P1) and the properties of $\aland$.
\end{proof}

What remains to be shown is that conditions (ALG1) and (ALG2) hold. To prove
these we proceed with a traditional algebraic soundness and completeness result
relying on a Lindenbaum-Tarski algebra.

For soundness it suffices to show that the axioms and rules of $\sbil$
all hold on bi-Heyting algebras.

\begin{proposition}[Soundness~\coqdoc{Alg.alg_soundness.html\#alg_soundness_sBIH}]\label{prop:sound}
If $\sderiv{\Gamma}{\phi}$ then $\{\gamma=\top\mid\gamma\in\Gamma\}\vDash_{\biha}\phi=\top$.
\end{proposition}

\begin{proof}
We reason by induction on the structure of the proof of $\Gamma\vdash\phi$.
The case for $\GHrule{El}$ is straightforward. 
The case for $\GHrule{Ax}$ consists in showing that in any bi-Heyting algebra all
axioms are interpreted as $\atop$~(\coqdoc{Alg.alg_soundness.html\#Axioms_one}). Using the antisymmetry of $\leq$
and the fact that $\atop$ is a top element, it is sufficient to show
all axioms are greater than $\atop$. The proof for all intuitionistic axioms are
as usual~(\coqdoc{Alg.Bi_Heyting_Algebras.html\#alg_A1},\coqdoc{Alg.Bi_Heyting_Algebras.html\#alg_A2},\coqdoc{Alg.Bi_Heyting_Algebras.html\#alg_A3},\coqdoc{Alg.Bi_Heyting_Algebras.html\#alg_A4},\coqdoc{Alg.Bi_Heyting_Algebras.html\#alg_A5},\coqdoc{Alg.Bi_Heyting_Algebras.html\#alg_A6},\coqdoc{Alg.Bi_Heyting_Algebras.html\#alg_A7},\coqdoc{Alg.Bi_Heyting_Algebras.html\#alg_A8},\coqdoc{Alg.Bi_Heyting_Algebras.html\#alg_A9},\coqdoc{Alg.Bi_Heyting_Algebras.html\#alg_A10}).
We are thus left with the bi-intuitionistic axioms $A_{11}$ to $A_{14}$ from Figure \ref{fig:Hilbert}, 
which we prove by instead proving the properties they correspond to in an algebra $\alg{A}\in\biha$.
We start with $A_{11}$~(\coqdoc{Alg.Bi_Heyting_Algebras.html\#alg_BA1}).
\begin{center}
\begin{tabular}{r c l l}
$a$ & $\leq$ & $b \alor (a\aexcl b)$ & (2 of Lem.~\ref{lem:bi-props}) \\
$a$ & $\leq$ & $(a\aexcl b) \alor b$ & (comm.~$\alor$) \\
$\atop\aland a$ & $\leq$ & $(a\aexcl b) \alor b$ & (absorp.~$\atop$) \\
$\atop$ & $\leq$ & $a\ato(a\aexcl b) \alor b$ & (ResI) \\
\end{tabular}
\end{center}
Then, we consider $A_{12}$~(\coqdoc{Alg.Bi_Heyting_Algebras.html\#alg_BA2}).
\begin{center}
\begin{tabular}{r c l l}
$\awneg\awneg(a\ato b)$ & $\leq$ & $a \ato b$ & (1 of Lem.~\ref{lem:bi-props}) \\
$a\aland\awneg\awneg(a\ato b)$ & $\leq$ & $a\aland (a \ato b)$ & (monot.~$\aland$) \\
$a\aexcl\awneg(a\ato b)$ & $\leq$ & $a\aland (a \ato b)$ & (trans.~$\leq$ and 3 of Lem.~\ref{lem:bi-props}) \\
$a\aexcl\awneg(a\ato b)$ & $\leq$ & $b$ & (trans.~$\leq$ and modus ponens~\coqdoc{Alg.Bi_Heyting_Algebras.html\#mp}) \\
$a$ & $\leq$ & $\awneg(a\ato b) \alor b$ & (ResE) \\
$a\aexcl b$ & $\leq$ & $\awneg(a\ato b)$ & (ResE) \\
$\atop\aland (a\aexcl b)$ & $\leq$ & $\awneg(a\ato b)$ & (absorp.~$\atop$) \\
$\atop$ & $\leq$ & $(a\aexcl b)\ato\awneg(a\ato b)$ & (ResI) \\
\end{tabular}
\end{center}
Next, we turn to $A_{13}$~(\coqdoc{Alg.Bi_Heyting_Algebras.html\#alg_BA3}).
\begin{center}
\begin{tabular}{r c l l}
$a$ & $\leq$ & $(b\alor c)\alor (a\aexcl(b\alor c))$ & (1 of Lem.~\ref{lem:bi-props}) \\
$(a\aexcl b)\aexcl c$ & $\leq$ & $a\aexcl(b\alor c)$ & (ResE) \\
$\atop\aland ((a\aexcl b)\aexcl c)$ & $\leq$ & $a\aexcl(b\alor c)$ & (absorp.~$\atop$) \\
$\atop$ & $\leq$ & $((a\aexcl b)\aexcl c)\ato(a\aexcl(b\alor c))$ & (ResI) \\
\end{tabular}
\end{center}
The last axiom to consider is $A_{14}$~(\coqdoc{Alg.Bi_Heyting_Algebras.html\#alg_BA4}).
\begin{center}
\begin{tabular}{r c l l}
$a\aexcl b$ & $\leq$ & $\aneg\aneg(a\aexcl b)$ & (double negation introduction~\coqdoc{Alg.Bi_Heyting_Algebras.html\#double_neg}) \\
$a$ & $\leq$ & $\aneg\aneg(a\aexcl b)\alor b$ & (ResE) \\
$a$ & $\leq$ & $\aneg(a\aexcl b)\ato b$ & (As $\aneg p \alor q \leq p \ato q$)) \\
$\aneg(a\aexcl b)\aland a$ & $\leq$ & $b$ & (ResI) \\
$(\atop\aland \aneg(a\aexcl b))\aland a$ & $\leq$ & $b$ & (absorp.~$\atop$) \\
$\atop$ & $\leq$ & $\aneg(a\aexcl b)\ato(a\ato b)$ & (ResI) \\
\end{tabular}
\end{center}

With all axioms considered, we are left with the case where the proof was
built using $\GHrule{MP}$ or $\GHrule{sDN}$. We focus on the latter as 
the former is as in the intuitionistic case.
Assume that $\int{\val}(\phi) = \atop$, then, as $\atop \aexcl \atop = \abot$, $\atop \aexcl \int{\val}(\phi) = \abot$,
hence, as $\abot \ato \abot = \atop$, $\awneg \int{\val}(\phi) \ato \abot = \atop$ and
so $\int{\val}(\neg\wneg \phi) = \atop$. 
\end{proof}

To prove completeness we use \emph{Lindenbaum-Tarski algebras}, i.e.~algebras whose
points are classes of formulas which are \emph{$\log{i}$-equivalent modulo $\Gamma$}
for a given $\Gamma$.

\begin{definition}
Let $\{\phi,\psi\}\cup\Gamma\subseteq\form$ and $i\in\{w,s\}$.
We say that $\phi$ and $\psi$ \emph{$\log{i}$-equivalent modulo $\Gamma$} if
$\ideriv{\Gamma}{\phi\leftrightarrow\psi}$.
The set of formulas $\log{i}$-equivalent modulo $\Gamma$ to $\phi$ is denoted
by $\eqprv{\phi}_{\Gamma}^{\log{i}}$~(\coqdoc{Alg.sBIH_alg_completeness.html\#eqprv}).
The class of all sets $\eqprv{\phi}_{\Gamma}^{\log{i}}$ is denoted
$\eqcla{i}{\Gamma}$.
When $\log{i}$ and $\Gamma$ are clear from context, we drop their mention and
talk about \emph{equivalent} formulas and write $\eqprv{\phi}$ and $\eqcla{}{}$.
\end{definition}

For the remaining of our discussion of the completeness proof let us 
fix a set of formulas $\Gamma$ and let us reiterate that $\sbil$ is the logic
we focus on.
Given that we want our Lindenbaum-Tarski algebra to a be bi-Heyting algebra, we need to
define an operator on equivalent classes for each and every algebraic operator
satisfying adequate properties.

Quite naturally, one would want to use logical connectives on representatives
of equivalence classes. 
For example, we can define the algebraic operator $\aland$
in our Lindenbaum-Tarski algebra as taking two equivalence classes $\eqprv{\phi}$ and
$\eqprv{\psi}$, with respective representative $\phi$ and $\psi$, and giving $\eqprv{\phi\land\psi}$.
Through this approach one needs to prove that the choice of representative does not
matter: by taking any $\phi'\in\eqprv{\phi}$ and $\psi'\in\eqprv{\psi}$, we must have
$\eqprv{\phi'\land\psi'}=\eqprv{\phi\land\psi}$.
While this result can be proved for all intuitionistic connectives in both $\wbil$ and 
$\sbil$~(\coqdoc{Alg.sBIH_alg_completeness.html\#epone},\coqdoc{Alg.sBIH_alg_completeness.html\#epzero},\coqdoc{Alg.sBIH_alg_completeness.html\#epjoin},\coqdoc{Alg.sBIH_alg_completeness.html\#epmeet},\coqdoc{Alg.sBIH_alg_completeness.html\#eprpc}), interestingly we can
only prove this result for $\excl$ in $\sbil$~(\coqdoc{Alg.sBIH_alg_completeness.html\#epsubtr}) using the left monotonicity and right anti-monotonicity of $\excl$.%
\footnote{Crucially, left monotonicity and right anti-monotonicity of $\excl$ only hold in $\wbil$
when the left-hand side of the consecution is the empty set.
We exploit this to find a counterexample in $\wbil$ by taking $\Gamma:=\{p_1,p_2,q_1,q_2\}$.
It then suffices to notice that $p_1$ and $p_2$ are $\log{w}$-equivalent modulo $\Gamma$, 
that $q_1$ and $q_2$ also are, but that $\nowderiv{\Gamma}{(p_1\excl q_1)\to(p_2\excl q_2)}$.}

With these algebraic operators defined, one can show that the resulting algebraic
structure is indeed a bi-Heyting algebra.
More precisely, this structure satisfies all properties of 
bounded lattices (Definition~\ref{def:boundlat}: \coqdoc{Alg.sBIH_alg_completeness.html\#epjcomm},\coqdoc{Alg.sBIH_alg_completeness.html\#epassoc},\coqdoc{Alg.sBIH_alg_completeness.html\#epjabsorp},\coqdoc{Alg.sBIH_alg_completeness.html\#epmcomm},\coqdoc{Alg.sBIH_alg_completeness.html\#epmassoc},\coqdoc{Alg.sBIH_alg_completeness.html\#epmabsorp}, \coqdoc{Alg.sBIH_alg_completeness.html\#eplowest},\coqdoc{Alg.sBIH_alg_completeness.html\#epjcomm})
and the residuation properties of bi-Heyting algebras (Definition~\ref{def:ha-biha}: \coqdoc{Alg.sBIH_alg_completeness.html\#epresiduation},\coqdoc{Alg.sBIH_alg_completeness.html\#epdresiduation}).%
\footnote{Note that to prove the residuation of $\excl$, we need to use the item (3) of Lemma~\ref{lem:derivations},
which does not hold for $\wbil$.}
We can therefore define our Lindenbaum-Tarski algebra as a bi-Heyting algebra.

\begin{definition}[Lindenbaum-Tarski algebra]\label{def:LTalg}
We define $\LT{s}{\Gamma}=(\eqcla{s}{\Gamma},\atop,\abot,\aland,\alor,\ato,\aexcl)$
to be the bi-Heyting algebra with the operators defined below.
\begin{center}
\begin{tabular}{c @{\hspace{1.5cm}} c @{\hspace{1.5cm}} c}
$\atop:=\eqprv{\top}_\Gamma$ &
$\eqprv{\phi}_\Gamma\aland\,\eqprv{\psi}_\Gamma:=\eqprv{\phi\land\psi}_\Gamma$ &
$\eqprv{\phi}_\Gamma\ato\,\eqprv{\psi}_\Gamma:=\eqprv{\phi\to\psi}_\Gamma$ \\
$\abot:=\eqprv{\bot}_\Gamma$ &
$\eqprv{\phi}_\Gamma\alor\,\eqprv{\psi}_\Gamma:=\eqprv{\phi\lor\psi}_\Gamma$ &
$\eqprv{\phi}_\Gamma\aexcl\,\eqprv{\psi}_\Gamma:=\eqprv{\phi\excl\psi}_\Gamma$
\end{tabular}
\end{center}
\end{definition}

A crucial property of the Lindenbaum-Tarski algebra $\LT{s}{\Gamma}$ is that
formulas are interpreted as their equivalence class under the valuation mapping
propositional variables to their equivalence class.

\begin{lemma}[\coqdoc{Alg.sBIH_alg_completeness.html\#LindAlgrepres}]\label{lem:algtrulem}
Let $\val_{\scriptscriptstyle\LT{}{}}:Prop\to\eqcla{}{}$ be such that $\val_{\scriptscriptstyle\LT{}{}}(p)=\eqprv{p}$. We have that $\int{\val}_{\scriptscriptstyle\LT{}{}}(\phi)=\eqprv{\phi}$.
\end{lemma}

\begin{proof}
The proof is straightforward:
it goes by induction on the structure of $\phi$, and relies on 
definitional equalities to go through, like in the following $\excl$ case:
$\int{\val}_{\scriptscriptstyle\LT{}{}}(\psi\excl\chi)=\int{\val}_{\scriptscriptstyle\LT{}{}}(\psi)\aexcl\int{\val}_{\scriptscriptstyle\LT{}{}}(\chi)=\eqprv{\psi\excl\chi}$.
\end{proof}

With this result in hand, we can easily prove completeness.

\begin{proposition}[Completeness~\coqdoc{Alg.sBIH_alg_completeness.html\#alg_completeness_sBIH}]\label{prop:compl}
If $\{\gamma=\top\mid\gamma\in\Gamma\}\vDash_{\biha}\phi=\top$ then $\sderiv{\Gamma}{\phi}$.
\end{proposition}

\begin{proof}
We prove the contrapositive. So, assume $\nosderiv{\Gamma}{\phi}$.
By Lemma~\ref{lem:algtrulem} we get that $\int{\val}_{\scriptscriptstyle\LT{}{}}(\phi)=\eqprv{\phi}$ and $\int{\val}_{\scriptscriptstyle\LT{}{}}(\gamma)=\eqprv{\gamma}$
for all $\gamma\in\Gamma$.
Clearly, we have for any such $\gamma$ that $\eqprv{\gamma}=\eqprv{\top}=\atop$ as $\sderiv{\Gamma}{\gamma\leftrightarrow\top}$.
However, we also have that $\nosderiv{\Gamma}{\top\rightarrow\phi}$ because of our assumption $\nosderiv{\Gamma}{\phi}$.
Consequently, we must have that $\eqprv{\phi}\neq\eqprv{\top}$.
Therefore, we exhibited a bi-Heyting algebra and an interpretation on it witnessing
$\{\gamma=\top\mid\gamma\in\Gamma\}\not\vDash_{\biha}\phi=\top$
\end{proof}

As promised, using soundness and completeness for $\sbil$ we can straightforwardly
prove the remaining conditions of algebraizability (ALG1) and (ALG2).
In fact, (ALG1) is nothing but the combination of soundness and completeness.

\begin{proposition}[ALG1~\coqdoc{Alg.sBIH_algebraizable.html\#sBIH_Alg1}]\label{prop:alg3}
The following holds.
\begin{center}
\begin{tabular}{c @{\hspace{1.4cm}} c @{\hspace{1.4cm}} c}
$\sderiv{\Gamma}{\phi}$ & if and only if & $\{\gamma=\top\mid\gamma\in\Gamma\}\vDash_{\biha}\phi=\top$
\end{tabular}
\end{center}
\end{proposition}

\begin{proposition}[ALG2~\coqdoc{Alg.sBIH_algebraizable.html\#sBIH_Alg2}]\label{prop:alg3}
The following holds.
\begin{center}
\begin{tabular}{c @{\hspace{1.4cm}} c @{\hspace{1.4cm}} c}
$\Theta \vDash_{\biha} \epsilon = \delta$ & if and only if & 
$\sderiv{\{\theta_1\leftrightarrow\theta_2\mid (\theta_1=\theta_2)\in \Theta\}}{\epsilon\leftrightarrow\delta}$
\end{tabular}
\end{center}
\end{proposition}

\begin{proof}
We start by proving the left-to-right direction and assume $\Theta \vDash_{\biha} \epsilon = \delta$.
Using Proposition~\ref{prop:compl} we can reduce our goal to 
$$\{(\theta_1\leftrightarrow\theta_2)=\top\mid(\theta_1=\theta_2)\in \Theta\}\vDash_{\biha}(\epsilon\leftrightarrow\delta)=\top.$$
Through the antisymmetry of $\leq$, assuming that $\int{\val}(\theta_1\leftrightarrow\theta_2)=\atop$
is equivalent to assuming $\int{\val}(\theta_1)=\int{\val}(\theta_2)$. 
Therefore, assuming all the equations $\{(\theta_1\leftrightarrow\theta_2)=\top\mid(\theta_1=\theta_2)\in \Theta\}$
to hold amounts to assuming the set of equations $\Theta$. 
Using our initial hypothesis we get that $\int{\val}(\epsilon)=\int{\val}(\delta)$. 
It is then trivial to prove that the equation $(\epsilon\leftrightarrow\delta)=\top$ holds.

We turn to the remaining direction and assume
$\sderiv{\{\theta_1\leftrightarrow\theta_2\mid (\theta_1=\theta_2)\in \Theta\}}{\epsilon\leftrightarrow\delta}$.
By applying Proposition~\ref{prop:sound} to this assumption we obtain
$$\{(\theta_1\leftrightarrow\theta_2)=\top\mid (\theta_1=\theta_2)\in \Theta\}\vDash_{\biha}(\epsilon\leftrightarrow\delta)=\top.$$
Here again, we prove $\Theta \vDash_{\biha} \epsilon = \delta$ using the 
equivalence between the equations $\phi=\psi$ and $\phi\leftrightarrow\psi=\top$.
\end{proof}

Now that we have proved that all conditions hold for the class $\biha$ with defining equation
$p=\top$ and equivalence formula $p\leftrightarrow q$, we obtain our precise algebraizability result.

\begin{theorem}\label{thm:algbiha}
$\sbil$ is (finitely) algebraizable with respect to $\biha$ with defining equations
$\{p=\top\}$ and equivalence formulas $\{p\leftrightarrow q\}$.
\end{theorem}

While our presentation of bi-Heyting algebras in 
Definition~\ref{def:ha-biha} is not equational, 
such a definition can be given~\cite[1.1]{Rau74alg},
showing that bi-Heyting algebras form a variety. 
Thus, Theorem~\ref{thrm:easqv} allows us to conclude the following.

\begin{corollary}
  \label{cor:biheytingequivalg}
  The equivalent algebraic semantics of $\sbil$ are bi-Heyting algebras.
\end{corollary}


\subsection{$\wbil$ is equivalential}

Now that we have situated $\sbil$ in the Leibniz hierarchy, we investigate
$\wbil$ in the same way. 
While we find that $\wbil$ is not algebraizable, we show that it is equivalential but not finitely so.
In fact, the relationship between $\wbil$ and bi-Heyting algebras
is so weak that we find it to not even have an algebraic semantics over
$\biha$. 

We first prove that $\wbil$ is not algebraizable.
To show this, we use the Isomorphism theorem (Theorem~\ref{thrm:theisothrm}),
for which we need a characterisation of the $\wbil$-filters~(\coqdoc{Alg.wBIH_not_algebraizable.html\#wBIH_filter}) over bi-Heyting algebras:
they are lattice filters.

\begin{definition}[\coqdoc{Alg.wBIH_not_algebraizable.html\#lattice_filter}]
  For any bi-Heyting algebra $\alg{A}$, a \emph{lattice filter} is a subset $F
  \subseteq \alg{A}$ such that:
  \begin{enumerate}
    \item $\atop \in F$; 
    \item If $a \in F$ and $a \leq b$ then $b \in F$;
    \item If $a,b \in F$ then $a \aland b \in F$.
  \end{enumerate}
\end{definition}

This is to be expected as, for Heyting algebras, intuitionistic
filters are the lattice filters~\cite[Exercise 2.31]{Fon16}.

\begin{lemma}
  \label{lem:vdashfiltfilt}
  For any bi-Heyting algebra $\alg{A}$ and subset $F \subseteq\alg{A}$, 
  the following conditions are equivalent:
  \begin{enumerate}
    \item $F$ is a $\wbil$-filter;
    \item $F$ is a lattice filter.
  \end{enumerate}
\end{lemma}
\begin{proof}
  We prove the equivalence by showing separately that each statement implies the other.

  \begin{description}
    \item[($1 \Rightarrow 2$)] Let us first show that if $F$ is a $\wbil$-filter
      it is then a lattice filter~(\coqdoc{Alg.wBIH_not_algebraizable.html\#d_to_l_filter}). 

      We start by showing that $\atop\in F$.
      As $F$ is a $\wbil$-filter and $\wderiv{}{\top}$, we get that $\int{\val}(\top)\in F$
      for any valuation $\int{\val}$ over $\alg{A}$. But $\int{\val}(\top)=\atop$ by definition,
      so $\atop\in F$.

      Next, we show that $F$ is closed under $\leq$.
      Let $a,b \in F$ such that $b \geq a$.
      Then, $a \ato b = \atop$, hence $a \ato b\in F$ given the above.
      Let $\int{\val}$ be a valuation over $\alg{A}$ such that $\int{\val}(p) = a$ and $\int{\val}(q) = b$.
      As $\wderiv{p, p \to q}{q}$ and $\int{\val}(p),\int{\val}(p)\ato\int{\val}(q) \in F$,
      then by the fact that $F$ is a $\wbil$-filter we get $\int{\val}(q) \in F$ and hence $b \in F$.

      Finally, we show that $F$ is closed under (finite) meets.
      Let $a,b \in F$, and let $\int{\val}$ be a valuation over $\alg{A}$ such that $\int{\val}(p) = a$ and $\int{\val}(q) = b$.
      Then, as $\wderiv{p,q}{p\land q}$ and $\int{\val}(p),\int{\val}(q)\in  F$, 
      we get $\int{\val}(p) \aland \int{\val}(q) = \int{\val}(p \land q) \in F$ as it is a $\wbil$-filter.
      Consequently, $a \aland b \in F$.

      We conclude that $F$ is a lattice filter.
      
    \item[($2 \Rightarrow 1$)] We now show that if $F$ be a lattice filter it is then a
      $\wbil$-filter~(\coqdoc{Alg.wBIH_not_algebraizable.html\#l_to_d_filter}).
      We need to show that $F$ is closed under consecutions: if $\wderiv{\Gamma}{\phi}$
      then for any valuation $\val$ over $\alg{A}$, if $\int{\val}(\Gamma)\subseteq F$ then
      $\int{\val}(\phi)\in F$.
      Assume that $\wderiv{\Gamma}{\phi}$. We prove our goal by induction on the derivation of $\wderiv{\Gamma}{\phi}$.
      In each case, we let $\val$ be a valuation such that $\int{\val}(\Gamma)\subseteq F$ and we show $\int{\val}(\phi)\in F$.

      The case for the rule \GHrule{El} is straightforward, and the one for \GHrule{Ax} can be treated as in
      Proposition~\ref{prop:sound}, where it is shown that $\int{\val}(\psi) = \atop$ for every
      axiom $\psi$, as $\atop\in F$ given that $F$ is a lattice filter.

      For the case of \GHrule{MP}, we get $\int{\val}(\psi) \in F$ and $\int{\val}(\psi \to \theta) \in F$
      by induction hypothesis.
      By definition, we get $\int{\val}(\psi) \ato \int{\val}(\theta) \in F$.
      So, $\int{\val}(\psi) \aland (\int{\val}(\psi) \ato h(\theta))\in F$ as $F$ is a lattice filter.
      As $\int{\val}(\psi) \aland (\int{\val}(\psi) \ato \int{\val}(\theta)) \leq \int{\val}(\theta)$,
      we get $\int{\val}(\theta) \in F$, solving the case for \GHrule{MP}.

      Let us now show the case for \GHrule{wDN}. 
      It suffices to show that $\aneg\awneg \int{\val}(\psi) = \atop$,
      as $\atop\in F$ given that $F$ is a lattice filter.
      We prove this equality via the antisymmetry of $\leq$, and as
      $\atop$ is the top element we are left to prove that
      $\atop \leq \aneg\awneg \int{\val}(\psi)$.
      Using (ResI) we reduce this to $\atop\aland\awneg \int{\val}(\psi) \leq \abot$,
      which is equivalent to $\awneg \int{\val}(\psi) \leq \abot$.
      In turn, using (ResE) we further reduce our goal to $\atop\leq \int{\val}(\psi)\alor\abot$,
      equivalent to $\atop\leq \int{\val}(\psi)$.
      This result can be obtained by the fact that $\sbil$ is an extension of $\wbil$~(\coqdoc{GenHil.BiInt_extens_interactions.html\#sBIH_extens_wBIH}),
      giving us $\sderiv{}{\psi}$ which entails $\int{\val}(\psi)=\atop$ by soundness of $\sbil$
      (Proposition~\ref{prop:sound}).
      Therefore, $F$ is a $\wbil$ filter. \qedhere
  \end{description}
\end{proof}

We are now able to prove that $\wbil$ is not algebraizable, inspiring ourselves from~\cite[Example 3.61]{Fon16}.

\begin{theorem}
  \label{thrm:weakbiintalg}
  Weak bi-intuitionistic logic is not algebraizable.%
\footnote{The proof of this theorem has not been fully formalised,
as it crucially relies on the Isomorphism theorem, which we have not formalised due to its generality.
Still, the counterexample in the proof, i.e.~the three elements chain, and its properties 
are all formalised.}
\end{theorem}
\begin{proof}
  Assume for a contradiction that $\wbil$ is algebraizable.
  Then by the Isomorphism Theorem (Theorem~\ref{thrm:theisothrm}) and
  Lemma~\ref{lem:vdashfiltfilt}
  there exists an isomorphism between lattice filters of any bi-Heyting algebra $\alg{A}$ and
  relative congruences $Con_\clalg{K}(\alg{A})$. To show this is not true we give a
  counter-example to this isomorphism.

  We take $c_3 = \{0,1/2,1\}$, the three element chain, as our needed
  counter-example.

  \begin{center}
    \begin{tikzpicture}
      \node (1) at (0,0) {$0$};
      \node (2) at (0,1) {$1/2$};
      \node (3) at (0,2) {$1$};
      \draw (1) -- (2);
      \draw (2) -- (3);
    \end{tikzpicture}
  \end{center}

  We define $a \ato b = 1$ if $a \leq b$ and $a \ato b =b$ otherwise.
  Additionally, we define $a \aexcl b = 0$ if $a\leq b$ and $a \aexcl b =a$ otherwise.
  It is known that this definition of $\ato$ forms a Heyting algebra, so let us show
  that our $\aexcl$ makes it into a bi-Heyting algebra~(\coqdoc{Alg.wBIH_not_algebraizable.html\#zho_alg}). 

  Assume that $a \aexcl b \leq c$.
  If it is the case that $a \leq b$, then $a \leq b \alor c$ is easily obtained.
  Instead, if $a \not \leq b$ then we have $a = a\aexcl b \leq c \leq b\alor c$.
  In the other direction, if $a \leq b \alor c$, and $a \leq b$,
  then $a \aexcl b = 0 \leq c$.
  In the case where $a \not \leq b$, we get $a\aexcl b = a$.
  Therefore it is sufficient to show $a\leq c$.
  Given that $a \not \leq b$ and $a \leq b \alor c$, 
  we can easily obtain this result.
  Consequently, $c_3$ is a bi-Heyting algebra.

  Our next task consists in showing that $c_3$ is simple, i.e.~it has only two congruences~(\coqdoc{Alg.wBIH_not_algebraizable.html\#only_two_zho_congruences}).

  Let $C$ be any congruence on $c_3$.
  Either $C = \Delta = \{(a,a) \mid a \in c_3\}$ or $C \neq\Delta$.
  If $C \neq \Delta$, there exist $a,b \in c_3$ such that $(a,b) \in C$ and $a \neq b$.
  There are three cases for the given $a$ and $b$:
  \begin{enumerate}
    \item If $(0,1) \in C$, then $(0\aland 1/2, 1 \aland 1/2) = (0,1/2) \in C$ and 
      so $C$ is the trivial congruence.
    \item If $(0,1/2) \in C$, then $(0 \ato 0, 1/2 \ato 0) = (1,1/2)$ and so $C$
      is the trivial congruence.
    \item If $(1/2,1) \in C$, then $(1 \aexcl 1/2, 1 \aexcl 1) = (1,0) \in C$ and so
      $C$ is the trivial congruence.
  \end{enumerate}
  So, we can see that if $C \neq \Delta$, then $C$ is the trivial congruence.
  As a consequence, the cardinality of the relative congruences of $\alg{A}$ must be less
  then or equal to $2$.

  As all three distinct upsets on $c_3$ are lattice filters~(\coqdoc{Alg.wBIH_not_algebraizable.html\#first_zho_lattice_filters},\coqdoc{Alg.wBIH_not_algebraizable.html\#second_zho_lattice_filters},\coqdoc{Alg.wBIH_not_algebraizable.html\#third_zho_lattice_filters}), this means there cannot be an
  isomorphism between the relative congruences of $\alg{A}$ and the $\wbil$-filters
  of $c_3$, therefore $\wbil$ is not algebraizable.
\end{proof}

We now show that $\wbil$ is equivalential with equivalence formulas
$\Delta(x,y) = \{(\neg\wneg)^n (x \leftrightarrow y) \mid n\in \mathbb{N}\}$.

\begin{theorem}
  \label{thrm:wbilequiv}
  $\wbil$ is equivalential with equivalence formulas $\{(\neg\wneg)^n (x \leftrightarrow y) \mid n\in \mathbb{N}\}$.
\end{theorem}
\begin{proof}
  To prove $\wbil$ is equivalential we need to prove that $\Delta(x,y)$ satisfies
  the three properties (R), (MP') and (Re) from Definition~\ref{def:equivalential}.
  
  First, we prove (R)~(\coqdoc{Alg.wBIH_equivalential.html\#wBIH_R}). Note that $\wderiv{}{x \leftrightarrow y}$.
  Thus, by repeated applications of \GHrule{wDN}, we easily obtain
  $\wderiv{}{(\DN)^n (x \leftrightarrow y)}$ for all $n$, hence $\vdash_w \Delta(x,x)$.

  Second, we focus on (MP')~(\coqdoc{Alg.wBIH_equivalential.html\#wBIH_MP}).
  Note that $\wderiv{x,x \to y}{y}$, which straightforwardly gives us $\wderiv{x,x \leftrightarrow y}{y}$.
  Thus, by monotonicity we get $\wderiv{x,\Delta(x,y)}{y}$.
  
  Third and last, we prove (Re) for every operator in $\wbil$.
  For the intuitionistic part we only show the cases of $\top$ and $\to$,
  as all other cases can be treated similarly (\coqdoc{Alg.wBIH_equivalential.html\#wBIH_Re_Bot},\coqdoc{Alg.wBIH_equivalential.html\#wBIH_Re_Or},\coqdoc{Alg.wBIH_equivalential.html\#wBIH_Re_And}).
  And we also show the special case of $\excl$.
  \begin{enumerate}
    \item[($\top$)] (\coqdoc{Alg.wBIH_equivalential.html\#wBIH_Re_Top}) As we have proven $\wderiv{}{\Delta(x,x)}$ for all $x$,
      we get $\wderiv{}{\Delta(\top,\top)}$ by structurality.
      
    \item[($\to$)] (\coqdoc{Alg.wBIH_equivalential.html\#wBIH_Re_Imp}) To show $\wderiv{\Delta(x_1,x_2),\Delta(y_1,y_2)}{\Delta((x_1 \to y_1),(x_2\to y_2))}$,
    we let $n$ be a natural number and prove
    $\wderiv{\Delta(x_1,x_2),\Delta(y_1,y_2)}{(\DN)^n((x_1 \to y_1) \leftrightarrow (x_2\to y_2))}$.
    By monotonicity and the deduction theorem, it suffices to prove
    $$\wderiv{}{(\DN)^n(x_1 \leftrightarrow x_2) \to (\DN)^n(y_1\leftrightarrow y_2)\to(\DN)^n((x_1 \to y_1) \leftrightarrow (x_2\to y_2))}.$$
    One can show that $(\DN)^n$ distributes over $\to$~(\coqdoc{GenHil.sBIH_meta_interactions.html\#sDN_form_dist_imp}), like a normal modality, 
    so we can further reduce our goal to $\wderiv{}{(\DN)^n((x_1 \leftrightarrow x_2) \to (y_1\leftrightarrow y_2)\to ((x_1 \to y_1) \leftrightarrow (x_2\to y_2)))}$.
    This goal is reached via $n$ applications of \GHrule{wDN} on $\wderiv{}{(x_1 \leftrightarrow x_2) \to (y_1\leftrightarrow y_2)\to ((x_1 \to y_1) \leftrightarrow (x_2\to y_2))}$.
    As $n$ is arbitrary, we reached our goal.
      
    \item[($\excl$)] (\coqdoc{Alg.wBIH_equivalential.html\#wBIH_Re_Excl}) We proceed similarly to the above,
    but instead of reducing our goal to 
    $$\wderiv{}{(x_1 \leftrightarrow x_2) \to (y_1\leftrightarrow y_2)\to ((x_1 \excl y_1) \leftrightarrow (x_2\excl y_2))},$$
    which is provable in neither $\wbil$ nor $\sbil$, we reduce it to
    $$\wderiv{}{\DN(x_1 \leftrightarrow x_2) \to \DN(y_1\leftrightarrow y_2)\to ((x_1 \excl y_1) \leftrightarrow (x_2\excl y_2))}.$$
    This can be done by singling out $(\DN)^{n+1}(x_1\leftrightarrow x_2)$ and $(\DN)^{n+1}(y_1\leftrightarrow y_2)$ in
    $\Delta(x_1,x_2)$ and $\Delta(y_1,y_2)$ via monotonicity.
    We prove our goal using completeness with respect to the local consequence relation in the Kripke semantics (Theorem~\ref{thm:soundcomplkripke}).
    Let $\mo{M}$ be a model, and $w$ a world. 
    It suffices to show that $\mo{M},w\Vdash(x_1 \excl y_1) \leftrightarrow (x_2\excl y_2)$
    while assuming $\mo{M},w\Vdash \DN(x_1\leftrightarrow x_2)$ and
    $\mo{M},w\Vdash \DN(y_1\leftrightarrow y_2)$, as the validity of implications can be proved
    locally in the Kripke semantics.
    We only show $\mo{M},w\Vdash(x_1 \excl y_1) \to (x_2\excl y_2)$ as the other direction is symmetric.
    Let $v\in W$ such that $w\leq v$ and $\mo{M},v\Vdash x_1 \excl y_1$. 
    We show that $\mo{M},v\Vdash x_2\excl y_2$.
    We must therefore have a $u\leq v$ such that $\mo{M},u\Vdash x_1$ but $\mo{M},u\not\Vdash y_1$.
    However, as $w\leq v\geq u$ and $\mo{M},w\Vdash \DN(x_1\leftrightarrow x_2)$ and $\mo{M},w\Vdash \DN(y_1\leftrightarrow y_2)$,
    we get that $\mo{M},u\Vdash x_2$ but $\mo{M},u\not\Vdash y_2$.
    This shows that $\mo{M},v\Vdash x_2\excl y_2$, so we are done.
    \end{enumerate}

  Hence we can see $\wbil$ is an equivalential logic.
\end{proof}

The attentive reader may have noted that our set of equivalence formulas $\Delta(x,y)$ is
not finite. One could therefore wonder whether there is a finite set of equivalence formulas
with which $\wbil$ can be shown to be equivalential, thereby making it \emph{finitely}
equivalential. In the remaining of this subsection, we provide a negative answer to this question.

To prove this, we make use of the following proposition as given in 
\cite[Proposition 6.65.5 and Proposition 6.65.6]{Fon16} stating that
if a logic is finitely equivalential, then any set of equivalence formulas
for it can be finitarised.

\begin{proposition}
  \label{prop:equivmultiplecong}
  Let $\log{L}$ be some finite, finitary 
  equivalential logic with congruence formulas $\Delta(x,y)$. There exists a finite subset
  $\Delta'(x,y) \subseteq \Delta(x,y)$ such that $\Delta'(x,y) \dashv\vdash \Delta(x,y)$ and
  $\Delta'(x,y)$ is a set of congruence formulas for $\log{L}$.
\end{proposition}

We now prove that $\wbil$ is not finitely equivalential. Our proof
essentially show that
no finitarisation of $\Delta(x,y)$ can serve as set of equivalence
formulas for $\wbil$.

\begin{theorem}[\coqdoc{Alg.wBIH_equivalential.html\#wBIH_Re_Excl_fail}]
  $\wbil$ is not finitely equivalential.
\end{theorem}
\begin{proof}
  For a contradiction assume that $\wbil$ is finitely equivalential.
  By Proposition~\ref{prop:equivmultiplecong} there exists a finite subset $\Delta'(x,y)$ of $\Delta(x,y)$,
  which is a set of congruence formulas for $\wbil$.
  As $\Delta'(x,y)$ is finite, let $n$ be the largest $n$ such that $(\neg\wneg)^n(x \leftrightarrow y) \in \Delta'(x,y)$.
  Next, we show that the property (Re) fails for $\Delta'(x,y)$, therefore showing that it does not constitute
  a set of congruence formulas.
  More precisely, we show that $\nowderiv{\Delta'(x_1,x_2),\Delta'(y_1,y_2)}{\Delta'(x_1\excl y_1,x_2\excl y_2)}$.
  
  To do so, we exhibit a model and a point in it which forces $\Delta'(x_1,x_2)$ and $\Delta'(y_1,y_2)$, but not $\Delta'(x_1\excl y_1,x_2\excl y_2)$.
  The particular formula that our point fails to force is $(\neg\wneg)^n((x_1\excl y_1)\leftrightarrow (x_2\excl y_2))$.

  Consider the following Kripke model $\mo{B}$~(\coqdoc{Alg.wBIH_equivalential.html\#Xmas_lights}).
    \begin{center}
    \begin{tikzpicture}
    
    \draw (0,0) node[point](1a)[label=below:{$\scriptstyle 0$}]{};
    \draw (1.5,1) node[point](1b)[label=above:{$\scriptstyle1$}]{};
    \draw (3,0) node[point](2a)[label=below:{$\scriptstyle2$}]{};
    \draw (4.5,1) node[point](2b)[label=above:{$\scriptstyle3$}]{};
    \draw (6,0) node (empty1) {};
    \draw (7.5,0) node (empty2) {};
    \draw (6.75,0.3) node (dots1) {$\mbox{\fontsize{20}{30}\selectfont\(\cdots\)}$};
    \draw (9,1) node[point](x)[label=above:{$\scriptstyle 2n+1$}]{};
    \draw (10.5,0) node[point](y)[label=below:{$\scriptstyle 2n+2$}]{};
    \path[<-](1b)edge(1a);
    \path[<-](1b)edge(2a);
    \path[<-](2b)edge(2a);
    \path[<-](2b)edge(empty1);
    \path[<-](x)edge(empty2);
    \path[<-](x)edge(y);
    \end{tikzpicture}
  \end{center}

  Let $I$ be the interpretation of $\mo{B}$ such that 
  $I(x_1) = \{k \mid 0\leq k \leq 2n+2\}$ and 
  $I(x_2)=I(y_1)=I(y_2) = \{k \mid 0\leq k \leq 2n+1\}$.
  In essence, this interpretation makes $x_1$ true everywhere, and the other
  variables true everywhere but in $2n+2$. This has the effect that $x_1$ and $x_2$ are equivalent \emph{far enough}
  (from left to right) so that $\Delta'(x_1,x_2)$ and $\Delta'(y_1,y_2)$ hold in $0$,
  but $\Delta'(x_1\excl y_1,x_2\excl y_2)$ does not as we have
  $2n\not\Vdash(x_1\excl y_1)\leftrightarrow (x_2\excl y_2)$ and hence
  $0\not\Vdash(\DN)^n((x_1\excl y_1)\leftrightarrow (x_2\excl y_2))$.
  We can see that $2n\not\Vdash(x_1\excl y_1)\leftrightarrow (x_2\excl y_2)$ as
  we have $2n+1\Vdash(x_1\excl y_1)$ by the valuation in $2n+2$ making $x_1$ true but not $y_1$,
  while $2n+1\Vdash(x_2\excl y_2)$ again because of the valuation.
  Consequently, we have $\Delta'(x_1,x_2),\Delta'(y_1,y_2)\not\locmodels\Delta'(x_1\excl y_1,x_2\excl y_2)$
  hence $\nowderiv{\Delta'(x_1,x_2),\Delta'(y_1,y_2)}{\Delta'(x_1\excl y_1,x_2\excl y_2)}$ by soundness.

  This contradicts our initial assumption on the existence of a finite set of
  congruence formulas.
\end{proof}


\subsection{wBIL has no algebraic semantics over bi-Heyting algebras}

We have shown that $\wbil$ does not have a strong relationship to bi-Heyting
algebras in the sense of algebraizability, but there is an obvious further
question: does it have any relationship at all? The answer, from an abstract
algebraic logic perspective, is no. There is no algebraic semantics that can be
constructed on bi-Heyting algebras for $\wbil$.

Our proof of this statement makes a detour through the Kripke semantics%
\footnote{The idea on which this proof is based was discussed in private communications between
Shillito and Moraschini, which later influenced~\cite[Corollary 9.7]{Mor22}.}%
,
via the construction of a model consisting of a certain arrangement of $2n$
copies of another given model $\mo{M}$.

\begin{definition}[\coqdoc{Alg.wBIH_no_biHA_alg_sem.html\#Xmas_lightsM}]\label{def:modelconstruction}
  Let $\mo{M} = (X,\leq,I)$ be a Kripke model.
  For any $k\in\N$ we define $\mo{M}_k=(X_k,\leq_k,I_k)$ to be a copy of $\mo{M}$
  annotated by $k$, i.e.~where $X_k=\{(x,k)\mid x\in X\}$, $(x,k)\leq_k (y,k)$ if and only if $x\leq y$,
  and $I_k(p) = \{(x,k) \mid x \in I(p)\}$. 
  Now, for $n\in\N$ and $w\in X$, we define 
  $\mo{M}^n_w = (\{c\} \cup \bigcup\limits_{0 \leq k \leq 2n} X_k, \leq', I')$ such that:
  \begin{itemize}
  \item $\leq'$ is the least preorder such that $\leq' \supseteq \bigcup\limits_{0 \leq i \leq n} \leq_i$ and $(w,0) \leq' (w,1) \geq' ... \leq'
  (w,2n-1) \geq' (w,2n)\leq' c$;
  \item $I'(p) = \{c\} \cup \bigcup_{0 \leq k \leq 2n} I_k(p)$ for any $p\in\Prop$.
  \end{itemize}
\end{definition}

We can depict the
model $\mo{M}^n_w$ constructed in the definition above as follows.%
\footnote{In the formalisation we used yet another copy $\mo{M}_{n+1}$
of $\mo{M}$ instead of $c$. The only difference in this last copy is that we
modify the interpretation $I_{n+1}$ to make all propositional variables
everywhere in $X_{n+1}$.}
  \begin{center}
    \begin{tikzpicture}
    
    \draw (0,0) node[point](v)[label=below:{$\scriptstyle (w,0)$}]{};
    \draw (0,0) circle (0.85cm);
    \draw (0,1.1) node (F1) {$\mo{M}_0$};
    \draw (2,1) node[point](u)[label=above:{$\scriptstyle(w,1)$}]{};
    \draw (2,1) circle (0.85cm);
    \draw (2,-0.1) node (F2) {$\mo{M}_{1}$};
    \draw (4,0) node (empty1) {};
    \draw (6,0) node (empty2) {};
    \draw (5,0.3) node (dots1) {$\mbox{\fontsize{20}{30}\selectfont\(\cdots\)}$};
    \draw (8,1) node[point](x)[label=above:{$\scriptstyle(w,2n-1)$}]{};
    \draw (8,1) circle (0.85cm);
    \draw (8,-0.1) node (Fn) {$\mo{M}_{2n-1}$};
    \draw (10,0) node[point](y)[label=below:{$\scriptstyle(w,2n)$}]{};
    \draw (10,0) circle (0.85cm);
    \draw (10,1.1) node (Fu) {$\mo{M}_{2n}$};
    \draw (12,1) node[point](w)[label=below:{$c$}]{};
    \path[<-](w)edge(y);
    \path[<-](u)edge(v);
    \path[<-](u)edge(empty1);
    \path[<-](x)edge(empty2);
    \path[<-](x)edge(y);
    \end{tikzpicture}
  \end{center}

The model is constructed as $n$ zig-zags $\zz$ of copies of $\mo{M}$.
It is built so that $w$ in $\mo{M}$ and $(w,0)$ in $\mo{M}^n_w$
are $2n$-bisimilar (Definition~\ref{def:nbisim}) and therefore force the same formulas
of bi-depth $2n$ (Lemma~\ref{lem:depthnbisim}).
Notably, $w$ and $(w,0)$ force the same formulas of the shape $(\DN)^n\phi$ 
where $d(\phi)=0$.
However, these two points are crucially not $2n+1$ bisimilar,
as formulas of bi-depth $2n+1$ can inspect $c$ which has no
counterpart in $\mo{M}$.
We support our claim of $2n$-bisimilarity below.

\begin{lemma}[\coqdoc{Alg.wBIH_no_biHA_alg_sem.html\#w_Xmas_lightsM_nbisim_M}]\label{lem:constructionbisim}
  For any model $\mo{M}$ and world $w \in \mo{M}$, we have $\mo{M},w \bisim_{2n} \mo{M}^n_w,(w,0)$.
\end{lemma}

\begin{proof}
We only provide an intuitive argument, and refer to our formalisation for a more detailed proof.
In essence, moves along $\leq'$ in $\mo{M}^n_w$ can always be mimicked 
in $\mo{M}$ \emph{as long as we have not moved to $c$}.
The reason for this is that $c$ has not corresponding world in $\mo{M}$,
because of how valuation is defined on it. 
Therefore, we get that $w$ and $(w,0)$ are $k$-bisimilar as long as
there is not path allowing us to move from $(w,0)$ to $c$ in $k$ steps.
This is why we can prove $\mo{M},w \bisim_{2n} \mo{M}^n_w,(w,0)$:
the shortest path from $(w,0)$ to $c$ is via all $(w,i)$ for $1\leq i\leq 2n$,
which is of length $2n+1$,
hence all paths of length smaller than $2n$ are safe.
\end{proof}

Our proof that $\wbil$ has no algebraic semantics over bi-Heyting algebras heavily
relies on the construction outlined above, but also on two additional lemmas.

The first lemma is of a general nature, and is given in \cite[Theorem 2.16]{BloReb03} with a more
modern treatment in \cite[Proposition 3.3]{Mor22}. We did not formalise it in our
work, but instead relied on one of its instances within our proof.

\begin{lemma}
  \label{lem:prop33}
  If a logic $\log{L}$ has a $\tau(x)$-algebraic semantics, then for
  all $\epsilon(x) = \delta(x) \in \tau(x)$ and $\phi(q,\overline{p}) \in \form$, 
  $$x,\phi(\epsilon(x),\overline{p}) \dashv\vdash \phi(\delta(x),\overline{p}), x.$$
\end{lemma}

For the second lemma, we take inspiration from Moraschini~\cite[Proposition 9.6]{Mor22}.

\begin{lemma}[\coqdoc{Alg.wBIH_no_biHA_alg_sem.html\#alg_eq_iff_wBIH_eqprv}]
  \label{lem:tommolem}
  For any $\phi,\psi \in \form_{BI}$,
  $$
  \biha \vDash \phi = \psi \iff \phi \dashv\vdash_w \psi.
  $$
\end{lemma}
\begin{proof}
  First, assume that $\phi \dashv\vdash_w \psi$. 
  By the deduction theorem we get $\wderiv{}{\phi \to \psi}$ and $\wderiv{}{\psi \to \phi}$,
  which imply $\sderiv{}{\phi \to \psi}$ and $\sderiv{}{\psi \to \phi}$ as $\wbil$ and $\sbil$
  agree on theorems~(\coqdoc{GenHil.BiInt_extens_interactions.html\#sBIH_wBIH_same_thms}).
  By Theorem~\ref{thm:algbiha}, for every bi-heyting
  algebra $\alg{A}$ and valuation $\val$ over $\alg{A}$, 
  $\int{\val}(\phi \to \psi) = \atop$ and $\int{\val}(\psi \to \phi) = \atop$,
  hence $\int{\val}(\phi) \leq \int{\val}(\psi)$ and $\int{\val}(\psi) \leq \int{\val}(\phi)$.
  The last two statements imply $\int{\val}(\phi)=\int{\val}(\psi)$ via antisymmetry of $\leq$.
  As $\alg{A}$ and $\val$ are arbitrary, we proved $\biha\vDash \phi = \psi$. 

  One can prove the other direction by following the steps above, which are for most
  them equivalences except for Theorem~\ref{thm:algbiha} which needs to be
  replaced by Proposition~\ref{prop:sound}.
\end{proof}

We now have all the ingredients to prove that there is 
no algebraic semantics on $\biha$ for $\wbil$.
While the following proof is heavily inspired by Moraschini~\cite[Proposition 9.6]{Mor22},
our case required a serious alteration of the model construction used in his proof.%
\footnote{The model construction in Moraschini's proof consists in the addition of
a single extra world situated below a certain point in the initial model.
For our case we needed the more complex construction given in
Definition~\ref{def:modelconstruction}.}

\begin{theorem}[\coqdoc{Alg.wBIH_no_biHA_alg_sem.html\#No_wBIH_alg_sem}]
  \label{thm:notcompletealgebraiccounterpart}
  There is no $\tau$-algebraic semantics with respect to $\biha$ for which $\wbil$ is complete.
\end{theorem}
\begin{proof}
  Suppose, for a contradiction, that $\wbil$ is complete with respect to a $\tau$-algebraic semantics with respect to $\biha$. 

  We first establish that there must be an equation $\epsilon(x)=\delta(x)\in\tau$
  such that $\biha \not\vDash \epsilon(x) = \delta(x)$~(\coqdoc{Alg.wBIH_no_biHA_alg_sem.html\#Eq_elem_invalid}).
  If this is not the case, then we would have $\wderiv{}{x}$ by
  the Definition~\ref{def:taualgsem} of $\tau$-algebraic semantics and the
  assumption that $\wbil$ is complete for it.
  By structurality of $\wbil$, this implies that $\wderiv{}{\bot}$.
  But this is a contradiction, as $\wbil$ can be proved consistent
  via soundness with respect to the Kripke semantics (Theorem~\ref{thm:soundcomplkripke}).

  By Lemma~\ref{lem:tommolem} we can assume, without loss of generality, 
  that $\nowderiv{\epsilon(x)}{\delta(x)}$. 
  Thus there exists a model $\mo{M} = (W,\leq,I)$ and a world $w \in W$
  such that $\mo{M}, w \Vdash \epsilon(x)$ and $\mo{M}, w \not \Vdash \delta(x)$. 

  Now, we exploit the model $\mo{M}^{n}_w$ generated from $\mo{M}$ and $w$ following
  Definition~\ref{def:modelconstruction}, where $n = \max(d(\epsilon(x)), d(\delta(x)))$.
  By Lemma~\ref{lem:constructionbisim} and Lemma~\ref{lem:depthnbisim}, we get 
  $\mo{M}^n_w, (w,0) \Vdash \epsilon(x)$ and $\mo{M}^n_w, (w,0) \not \Vdash \delta(x)$.
  Given the zig-zag structure of $\mo{M}^n_w$, we can use Lemma~\ref{lem:jellyfish}
  to show that $\mo{M}^n_w, c \not\Vdash (\neg\wneg)^{n+1}(\epsilon(x) \to \delta(x))$,
  as $c$ can reach out to $(w,0)$ in $n+1$ steps of $\zz$,
  and $(w,0)$ witnesses the failure of the implication $\epsilon(x)\to\delta(x)$.
  Additionally, we get that $\mo{M}^n_w, c \Vdash x$ by definition of the valuation in $c$.
  Consequently, we have a proof of $x\not\globmodels(\neg\wneg)^{n+1}(\epsilon(x) \to \delta(x))$ which
  gives us $\nowderiv{x}{(\neg\wneg)^n(\epsilon(x) \to \delta(x))}$.
  As we additionally have that $\wderiv{}{(\neg\wneg)^{n+1} (\delta(x) \to \delta(x))}$
  via $n+1$ applications of the rule \GHrule{wDN}, we can infer 
  $\nowderiv{x,(\neg\wneg)^{n+1} (\delta(x) \to \delta(x))}{(\neg\wneg)^n(\epsilon(x) \to \delta(x))}$.
  But this is a contradiction, as Lemma~\ref{lem:prop33} informs us that this statement should hold
  if $\tau$ really is a $\tau$-algebraic semantics for $\wbil$.
  
  As our assumption on the existence of a $\tau$-algebraic semantics for which $\wbil$ is complete
  led to a contradiction, we know that such a $\tau$ cannot exist.
\end{proof}


\section{Leveraging bridge theorems}\label{sec:bridge}

In the previous section, we situated the logics $\sbil$ and $\wbil$ in the Leibniz hierarchy,
by showing that they were respectively finitely algebraizable and equivalential.
We can now harvest the fruits of our hard labour: we demonstrate the
use of the \AAL~approach by exploiting the so-called \emph{bridge theorems},
which equate properties of logics with properties of algebras.

While $\wbil$ does not have a strong link with its natural class of algebras $\biha$, 
the algebraizability of $\sbil$ with respect to $\biha$ shows that the connection
existing between the logic and the algebras is extremely tight.
We use this tight connection and bridge theorems for algebraizable logics 
to show properties of both $\sbil$ and $\biha$.

\subsection{Craig interpolation and amalgamation}

The first bridge theorem we inspect connects \emph{Craig interpolation}
to \emph{amalgamation}. We start by defining the notion of Craig interpolation
following Czelakowski and Pigozzi \cite[Definition 3.1]{CzePig99}.%
\footnote{Arguably this is not Craig interpolation in the usual sense.
Traditionally, Craig Interpolation is expressed as the existence of an
interpolant $\chi(\overline{q})$ of a provable implication
$\vdash\phi(\overline{p},\overline{q}) \to \psi(\overline{q},\overline{r})$, 
i.e.~such that $\vdash \phi(\overline{p},\overline{q}) \to \chi(\overline{q})$
and $\vdash \chi(\overline{q}) \to \psi(\overline{q}, \overline{r})$.
Still, Definition~\ref{def:craigint} can be seen as a language-agnostic
abstraction of the usual meaning of Craig Interpolation.}

\begin{definition}[Craig Interpolation]\label{def:craigint}
  A logic $\log{L}$ is said to have \emph{Craig interpolation} if for
  $\Gamma(\overline{p},\overline{q}) \vdash \phi(\overline{q},\overline{r})$ then 
  there exists a $\Gamma'(\overline{q})$ such that $\Gamma(\overline{p},
  \overline{q}) \vdash \Gamma'(\overline{q})$ and
  $\Gamma'(\overline{q}) \vdash \phi(\overline{q}, \overline{r})$.
\end{definition}

Next, we define the algebraic property called \emph{amalgamation}~\cite[Definition
5.2]{CzePig99}.

\begin{definition}[Amalgamation]
  A variety $\clalg{K}$ has the amalgamation property if, for any given pair of
  injective homomorphisms $f : \alg{C} \to \alg{A}$ and $g : \alg{C} \to \alg{B}$ with $\alg{A},\alg{B},\alg{C} \in\clalg{K}$ there
  exists a $\alg{D} \in\clalg{K}$ and injective homomorphisms $h : \alg{A} \to \alg{D}$ and $k : \alg{B} \to \alg{D}$
  such that $h \circ f = k \circ g$.
\end{definition}

The bridge theorem connecting these two notions can be found in the same work
of Czelakowski and Pigozzi~\cite[Theorem 3.6,Corollary 5.28]{CzePig99}.

\begin{theorem}\label{thm:bridgecraig}
If a logic $\log{L}$ is algebraizable with equivalent algebraic semantics $\clalg{K}$, then the following holds.
\begin{center}
\begin{tabular}{c c c}
$\log{L}$ has Craig interpolation & iff & $\clalg{K}$ has amalgamation.
\end{tabular}
\end{center}
\end{theorem}

Using this bridge, we can prove the following result about $\sbil$ and $\biha$.

\begin{corollary}
$\sbil$ has Craig interpolation and $\biha$ has amalgamation.
\end{corollary}

\begin{proof}
It was proved by Galli, Reyes and Sagastume~\cite{GallReySag03} that $\biha$ has strong amalgamation,
a property which implies amalgamation. So, we know that $\biha$ has
amalgamation, and using Theorem~\ref{thm:bridgecraig} we can conclude that
$\sbil$ has Craig interpolation.

An alternative, and inverse, proof starts from the fact that $\wbil$ has Craig interpolation
in the usual form~\cite{KowOno17,LyoTiuGorClou18}.
As $\wbil$ and $\sbil$ coincide on their theorems, which is where the usual notion of 
Craig interpolation is set, we therefore get that $\sbil$ has the usual Craig interpolation.
Now, we can show that Craig interpolation in the sense of Definition~\ref{def:craigint}
can be obtained using the usual sense.
Assume that $\sderiv{\Gamma(\overline{p},\overline{q})}{\phi(\overline{q},\overline{r})}$.
Then, as $\sbil$ is a finitary logic, there must be a finite subset 
$\Delta(\overline{p},\overline{q})\subseteq\Gamma(\overline{p},\overline{q})$ such that
$\sderiv{\Delta(\overline{p},\overline{q})}{\phi(\overline{q},\overline{r})}$.
Therefore, we can form the conjunction of all its elements and get
$\sderiv{\bigwedge\Delta(\overline{p},\overline{q})}{\phi(\overline{q},\overline{r})}$.
By Lemma~\ref{lem:derivations} we get that there is a $n\in\N$ such that
$\sderiv{}{(\DN)^n(\bigwedge\Delta(\overline{p},\overline{q}))\to\phi(\overline{q},\overline{r})}$.
Therefore we can use the usual Craig interpolation to get a $\chi(\overline{q})$ such that
$\sderiv{}{(\DN)^n(\bigwedge\Delta(\overline{p},\overline{q}))\to\chi(\overline{q})}$ and
$\sderiv{}{\chi(\overline{q})\to\phi(\overline{q},\overline{r})}$.
These last statements straightforwardly lead to
$\sderiv{\Delta(\overline{p},\overline{q})}{\chi(\overline{q})}$ and
$\sderiv{\chi(\overline{q})}{\phi(\overline{q},\overline{r})}$.
So, $\chi(\overline{q})$ is the $\Gamma'(\overline{q})$ we were looking for.
As Craig interpolation in the sense of Definition~\ref{def:craigint} holds for $\sbil$,
we therefore get that $\biha$ has amalgamation via Theorem~\ref{thm:bridgecraig}.
\end{proof}

\subsection{Right deductive uniform interpolation and coherence}

The second bridge theorem we consider relates \emph{right deductive uniform interpolation}
with \emph{coherence}, which we both define next.

\begin{definition}[Right deductive uniform interpolation]
A logic $\log{L}$ is said to have \emph{right deductive uniform interpolation} if for
any formula $\phi$ and $p\in\Prop$, there is a formula $\chi$ such that
(1) $p$ does not appear in $\chi$, 
(2) $\phi\vdash\chi$, and 
(3) for any formula $\psi$ not containing $p$ such that
$\phi\vdash\psi$ we have $\chi\vdash\psi$.
\end{definition}

\begin{definition}[Coherence]\label{def:coherence}
A variety $\clalg{V}$ is \emph{coherent}
if every finitely generated subalgebra of a finitely presented member
of $\clalg{V}$ is finitely presented.
\end{definition}

The bridge theorem we are interested in connects the both notions in the following way~\cite[Proposition 2.4]{KowMet19}.

\begin{theorem}\label{thm:bridgecoherence}
If a logic $\log{L}$ is algebraizable with equivalent algebraic semantics $\clalg{K}$, then the following holds.
\begin{center}
\begin{tabular}{c c c}
$\log{L}$ satisfies right deductive uniform interpolation & iff & $\clalg{K}$ is coherent and admits deductive interpolation.
\end{tabular}
\end{center}
\end{theorem}

While there are some elements we left undefined in Definition~\ref{def:coherence}
and Theorem~\ref{thm:bridgecoherence}%
\footnote{We refer the interested reader to \cite{vGoMetTsi17,KowMet19}.}%
, we want the reader to take away the following: if the class of algebras corresponding 
to an algebraizable logic is not coherent, then the logic fails right deductive uniform
interpolation.
We therefore obtain the following.

\begin{corollary}
The logic $\sbil$ fails right deductive uniform interpolation.
\end{corollary}

\begin{proof}
Given Theorem~\ref{thm:bridgecoherence}, it suffices to show that $\biha$ is not coherent.
The latter was proved by Kowalski and Metcalfe~\cite[Corollary 5.8]{KowMet19}.
So, $\sbil$ fails right deductive uniform interpolation.
\end{proof}

\subsection{Deduction-detachment theorem and equationally definable principal congruences}

Finally, we consider a bridge theorem bringing together the
\emph{deduction-detachment theorem} and the property of having
\emph{equationally definable principal congruences}.

The notion of deduction-detachment theorem we consider
here is a generalised language agnostic version of the traditional notion.
The following definition can be found in Font's textbook~\cite[Definition 3.76]{Fon16}.

\begin{definition}[Deduction-detachment theorem]
  A set of formulas $I(p,q) \subseteq \form$ in at most two variables is a
  \emph{Deduction-Detachment (DD) set} for a logic $\vdash$ when, for all
  $\Gamma \cup \{\phi,\psi\} \subseteq \form$, $$
  \Gamma, \phi \vdash \psi \quad \text{ iff } \quad \Gamma \vdash I(\phi,\psi).
  $$
  A logic \emph{satisfies the Deduction-Detachment Theorem} (DDT) when it has a
  DD set.
\end{definition}

To define the algebraic property corresponding to the DDT for
some algebraizable logics, we make use of a specific type of
congruence on an algebra $\alg{A}$ generated from a pair 
$(a,b)$ of elements of $\alg{A}$:
we define $\Omega^\alg{A}(a,b)$ to be the smallest congruence
on $\alg{A}$ that identifies $a$ and $b$.

\begin{definition}[Equationally definable principal congruences]
  A variety $\clalg{K}$ has \emph{equationally definable principal congruences} (EDPC),
  when there is a finite non-empty set of equations in four variables
  $\{\eta_i(x_0,x_1,y_0,y_1) = \zeta_i(x_0,x_1,y_0,y_1) \mid i = 1,...,m\}$ such that for all $\alg{A} \in \clalg{K}$ and
  $a,b,c,d \in \alg{A}$,
  $$
  (c,d) \in \Omega_\alg{A}(a,b) \iff \eta_{i}(a,b,c,d) = \zeta_i(a,b,c,d) \text{ for all } i.
  $$
\end{definition}

Standard examples of varieties having EDPC are Boolean and Heyting algebras, as
they satisfy the following equivalence:
$(c,d) \in \Omega_\alg{A}(a,b)$ if and only if 
$((a \ato b) \aland (b \ato a)) \aland c = ((a \ato b) \aland (b \ato a)) \aland d$.

We will now give the bridge theorem connecting the DDT and EDPC.
This is a famous result originally proven by Blok and Pigozzi~\cite{BloPig89},
with a textbook account given by Font~\cite[Corollary 3.86]{Fon16}.

\begin{theorem}\label{thm:bridgeedpc}
If a finitary logic $\log{L}$ is finitely algebraizable with equivalent
algebraic semantics $\clalg{K}$, then the following holds.
\begin{center}
\begin{tabular}{c c c}
$\log{L}$ satisfies DDT & iff & $\clalg{K}$ has EDPC.
\end{tabular}
\end{center}
\end{theorem}

Leveraging the theorem above, we prove that $\sbil$ does not have the DDT
by showing that $\biha$ does not have EDPC.
To do so we make use of the following characterisation of varieties of
bi-Heyting algebras that have the EDPC~\cite[Corollary 3.7]{Tay16}.

\begin{lemma}
  \label{lem:semisimpleedpc}
  Let $\clalg{K}$ be a variety of bi-Heyting algebras, then $\clalg{K}$ has the EDPC if and only
  if there exists an $n \in \N$ such that, for every $\alg{A}\in\clalg{K}$ and $a \in\alg{A}$ we have
  $(\aDN)^{n+1}(a) = (\aDN)^n(a)$.
\end{lemma}

Given this we can now see that $\sbil$ does not satisfy DDT.

\begin{theorem}
The logic $\sbil$ does not satisfy DDT.
\end{theorem}
\begin{proof}
  By Lemma~\ref{lem:semisimpleedpc} and Theorem~\ref{thm:bridgeedpc} it
  suffices to construct a bi-Heyting algebra $\alg{A}$ with an $a\in\alg{A}$ which satisfy
  $(\DN)^{n+1}(a) \neq (\DN)^n(a)$ for all $n\in\N$.

  Consider the Lindenbaum-Tarski algebra $\LT{s}{\emptyset}$.
  Given that $\nosderiv{}{(\DN)^np\to(\DN)^{n+1}p}$ for every $n$,
  we can conclude that $\eqprv{(\DN)^np}\neq\eqprv{(\DN)^{n+1}p}$,
  i.e.~$(\aDN)^{n}\eqprv{p}\neq(\aDN)^{n+1}\eqprv{p}$.
  This shows that $\biha$ does not have EDPC,
  hence $\sbil$ does not satisfy DDT.
\end{proof}

It was already known that $\sbil$ only satisfies a modified version of
the traditional deduction-detachment theorem, displayed in Lemma~\ref{lem:derivations}.
Still, the result above is new as it generalises this failure of the 
traditional deduction-detachment theorem, which is nothing but
an instance of the failure of the DDT for $\sbil$.


\section{Stepping aside \AAL: logics preserving (degrees of) truth}\label{sec:preserv}

The \AAL~analysis is clear: $\sbil$ is tightly connected to $\biha$, as it is algebraizable,
but $\wbil$ not at all, as it has no $\tau$-algebraic semantics over it.
One may be finding this state of affairs too harsh, as $\wbil$ and $\biha$ can surely
be connected in some way. 

In this section, we briefly provide another framework which gives an
algebraic semantic for both $\sbil$ and $\wbil$ over $\biha$.
More precisely, in this context 
$\sbil$ is the logic \emph{preserving truth} over $\biha$ 
while $\wbil$ is the logic \emph{preserving degrees of truth} over $\biha$
\cite{Now90}.

Let us start with $\sbil$.

\begin{definition}[Logic preserving truth]\label{def:lpt}
The logic preserving truth $\pt$ on a class $\clalg{K}$ of algebras with a top element $\atop$%
\footnote{We consider a restricted version of this notion which works for our case:
here we use $\atop$ in our notion, but one can generalise it using arbitrary designated point.}
is defined as
\begin{center}
\begin{tabular}{c c c}
$\Gamma\pt\phi$ & $:=$ & 
$\forall \alg{A}\in \clalg{K}.\,\forall \val.\,\;\text{if }(\forall \gamma\in\Gamma.\;\atop= \int{\val}(\gamma))\text{ then }\atop= \int{\val}(\phi).$
\end{tabular}
\end{center}
\end{definition}

By looking at the definition above, one can directly see that the statement 
that $\sbil$ is the logic preserving truth over $\biha$ is nothing but the combination
of the already established Proposition~\ref{prop:sound} and Proposition~\ref{prop:compl}.

\begin{corollary}
The logic $\sbil$ is the logic preserving truth over $\biha$.
\end{corollary}

Now, we focus $\wbil$. We can retrieve an algebraic semantics over $\biha$
for this logic via the notion of logic preserving degrees of truth.

\begin{definition}[Logic preserving degrees of truth \coqdoc{Alg.algebraic_semantic.html\#alg_tdconseq}]\label{def:ptd}
The logic preserving degrees of truth $\vdash_{pdt}$ on a class $\clalg{K}$ of algebras is defined as follows,
where $\Gamma'\subseteq_f\Gamma$ means $\Gamma'\subseteq\Gamma$ and $\Gamma'$ is finite.%
\footnote{Two remarks can be made about our definition.
First, as in Definition~\ref{def:lpt} we use a restricted definition of this notion which
assumes the definability of inequality in the algebras.
Second, the use of a finite $\Gamma'$ instead of $\Gamma$ may look
unnatural, as it builds in the compactness of the logic.
We decided to use this definition for two reasons.
First, we took inspiration from Remark 2.4 of \cite{Mor24} which
claims that one can focus their attention to the finitary case.
Second, the use of an infinite $\Gamma$ in this definition
would not allow our proof of Theorem~\ref{thm:pdt} below
to go through, as we shall explain later.}
\begin{center}
\begin{tabular}{c c c}
$\Gamma\vdash_{pdt}\phi$ & $:=$ & 
$\exists\Gamma'\subseteq_f\Gamma.\,\forall \alg{A}\in\clalg{K}.\,\forall \val.\,\forall a.\,\;\text{if }(\forall \gamma\in\Gamma'.\;a\leq \int{\val}(\gamma))\text{ then }a\leq \int{\val}(\phi)$
\end{tabular}
\end{center}
\end{definition}

To show that $\wbil$ is the logic preserving degrees of truth of $\biha$, 
we use a Lindenbaum-Tarski algebra $\LT{w}{\emptyset}$ for $\wbil$ built
as in Definition~\ref{def:LTalg}~(\coqdoc{Alg.wBIH_alg_completeness.html\#LindAlg}). 
Here, it is crucial that we use the empty set of formulas as context as
this restriction allows us to prove that $\LT{w}{\emptyset}\in\biha$
when defining $\aexcl$ in $\LT{w}{\emptyset}$ as the operator working on
representative of equivalence classes in $\eqcla{w}{\emptyset}$.
As already noted, if we allowed non-empty contexts, then defining the $\aexcl$ in this
way would fail to give us a bi-Heyting algebra.

As in $\LT{s}{\Gamma}$ we get that formulas in $\LT{w}{\emptyset}$ are
interpreted as their equivalence class.

\begin{lemma}[\coqdoc{Alg.wBIH_alg_completeness.html\#LindAlgrepres}]\label{lem:walgtrulem}
Let $\val_{\scriptscriptstyle\LT{}{}}:Prop\to\eqcla{w}{\emptyset}$ such that $\val_{\scriptscriptstyle\LT{}{}}(p)=\eqprv{p}$. We have that $\int{\val}_{\scriptscriptstyle\LT{}{}}(\phi)=\eqprv{\phi}$.
\end{lemma}

\begin{proof}
The proof is similar to the one for Lemma~\ref{lem:algtrulem}.
The only interesting point to note here is that while the latter
proof implicitly relied on monotonicity laws for $\excl$ \emph{in presence of a context $\Gamma$},
this proof relies on these laws in the \emph{empty} context.
The former laws hold only in $\sbil$, but the latter hold in both.
\end{proof}

Relying on the above, we straightforwardly prove $\wbil$ to be
the logic preserving degrees of truth of $\biha$.

\begin{theorem}\label{thm:pdt}
The logic $\wbil$ is the logic preserving truth over $\biha$.
\end{theorem}

\begin{proof}
Assume $\Gamma\pdt\phi$.
By definition, there is a finite $\Gamma'\subseteq\Gamma$ such that 
$$\forall \alg{A}\in \clalg{K}.\,\forall \val.\,\forall a.\,\;\text{if }(\forall \gamma\in\Gamma'.\;a\leq \int{\val}(\gamma))\text{ then }a\leq \int{\val}(\phi).$$
Given that $\Gamma'$ is finite, we can create the conjunction of all its elements $\bigwedge\Gamma'$.
Now, if in the statement above we set $\LT{w}{\emptyset}$ for $\alg{A}$, and $\val_{\scriptscriptstyle\LT{}{}}$ from Lemma~\ref{lem:walgtrulem} for $\val$, 
as well as $\eqprv{\bigwedge\Gamma'}$ for $a$, we get the following.%
\footnote{The finiteness of $\Gamma'$ allows us to form the conjunction of all its elements,
which is a formula and therefore has a corresponding equivalence class in $\LT{w}{\emptyset}$.
If instead we used an infinite $\Gamma$, we would not be able to perform this move, as we are not sure whether
the infinite conjunction $\bigwedge\Gamma$ has a corresponding element (with adequate properties) 
in $\LT{w}{\emptyset}$.
A way to reconcile a version of Definition~\ref{def:ptd} where $\Gamma$ is not reduced to a finite subset 
with a proof of Theorem~\ref{thm:pdt} could consist in considering the \emph{canonical extension}
of $\LT{w}{\emptyset}$, which provides adequate infinite meets~\cite{JonTar52,GerHar01,GerHarVen06,GerVos09,Ger14}.}
$$\text{if }(\forall \gamma\in\Gamma'.\;\eqprv{\bigwedge\Gamma'}\leq \int{\val}_{\scriptscriptstyle\LT{}{}}(\gamma))\text{ then }\eqprv{\bigwedge\Gamma'}\leq \int{\val}_{\scriptscriptstyle\LT{}{}}(\phi)$$
By properties of meet, we can easily show that 
$\forall \gamma\in\Gamma'.\;\eqprv{\bigwedge\Gamma'}\leq \eqprv{\gamma}$,
which gives us the antecedent of the implication above as $\int{\val}_{\scriptscriptstyle\LT{}{}}(\gamma)=\eqprv{\gamma}$
via Lemma~\ref{lem:walgtrulem}.
Therefore, we get $\eqprv{\bigwedge\Gamma'}\leq \int{\val}_{\scriptscriptstyle\LT{}{}}(\phi)$ hence $\eqprv{\bigwedge\Gamma'}\leq \eqprv{\phi}$
thanks to another application of Lemma~\ref{lem:walgtrulem}.
We easily obtain $\eqprv{\top}\leq \eqprv{(\bigwedge\Gamma')\to\phi}$ via properties of bi-Heyting algebras
and the definition of $\LT{w}{\emptyset}$.
Consequently, we have that $\eqprv{\top}$ and $\eqprv{(\bigwedge\Gamma')\to\phi}$ are the same equivalence
classes, which implies that $\wderiv{}{\top\to((\bigwedge\Gamma')\to\phi)}$ hence $\wderiv{}{(\bigwedge\Gamma')\to\phi}$.
Using the deduction-detachment theorem from Lemma~\ref{lem:derivations},
we get that $\wderiv{\bigwedge\Gamma'}{\phi}$ hence $\wderiv{\Gamma'}{\phi}$.
A last application of monotonicity (Definition~\ref{def:logic}) gives us our goal $\wderiv{\Gamma}{\phi}$.

The other direction, showing $\Gamma\pdt\phi$ from $\wderiv{\Gamma}{\phi}$ is a straightforward soundness argument.
\end{proof}

So, in this light $\wbil$ can be seen as having a semantics on the class of bi-Heyting algebras.
The downside, however, is that this type of semantics does not come equipped with the strength
of the \AAL~framework exhibited in Section~\ref{sec:bridge}.


\section{Conclusion}\label{sec:concl}

In this paper we laid formalised, hence solid, foundations to the algebraic treatment of bi-intuitionistic logics,
a field which had not been thoroughly revisited since the discoveries of mistakes in Rauszer's works.

More precisely, we formalised in the interactive theorem prover Rocq results situating the logics $\sbil$
and $\wbil$ in the Leibniz hierarchy, by showing that
the former is implicative and finitely algebraizable over the class $\biha$ of bi-Heyting algebras,
while the latter is not algebraizable, has no algebraic semantic over $\biha$, but is non-finitely equivalential.

We exploited the algebraizability of $\sbil$ to showcase the strength of the abstract algebraic logic framework:
we could show that $\sbil$ has Craig interpolation but lacks both right deductive uniform interpolation and
a deduction-detachment theorem, by using bridge theorems and known facts about $\biha$,
i.e.~that it has amalgamation but fails coherence and does not have equationally definable principal congruences.

Finally, we made a step aside the abstract algebraic logic framework to show that one
could still connect the logic $\wbil$ to $\biha$ by the former as the logic preserving
degrees of truth of the latter. In this context, $\sbil$ becomes the logic preserving truth of $\biha$.

\section*{Acknowledgements}

The authors would like to acknowledge (in no specific order) Miguel Martins, Tommaso Moraschini,
James G. Raftery, Nick Bezhanishvili, Christopher J. Taylor, Tomasz Kowalski,
Adam P\v renosil and Tom\' a\v s L\'avi\v cka for the many enlightening conversations
we had with them throughout this project.

\bibliographystyle{plain}
\bibliography{algbiint}

\end{document}